\documentclass[11pt]{article}

\usepackage{amsmath, amsfonts, amsthm,bm}
\usepackage{amssymb,graphics,epsfig}
\usepackage{enumerate}
\usepackage[all,color]{xy}    
\usepackage{color}         

\usepackage{geometry}
\newtheorem{defn}{Definition}
\newtheorem{thm}{Theorem}
\newtheorem{prop}{Proposition}
\newtheorem{lem}{Lemma}
\newtheorem{cor}[lem]{Corollary}
\newtheorem{rem}{Remark}

\DeclareMathOperator{\sign}{\operatorname{sign}}
\DeclareMathOperator{\erf}{erf}
\DeclareMathOperator{\esup}{ess-sup}

\let\ds\displaystyle
\let\bs\boldsymbol

\newcommand{\FF}{{\mathcal F}}

\newcommand{\MM}{{\mathcal M}}
\newcommand{\NN}{{\mathcal N}}

\newcommand{\PPP}{{\mathcal P}}
\newcommand{\LL}{{\mathcal L}}

\newcommand{\R}{ {\mathbb R} }

\newcommand{\N}{ {\mathbb N} }

\newcommand{\E}{ {\mathbb E} }
\newcommand{\V}{ {\mathbb V} }
\newcommand{\Prob}{ {\mathbb P} }
\newcommand{\Q}{ {\mathbb Q} }

\newcommand{\ep}{{\varepsilon}}
\newcommand{\indiq}{\hbox{\rm 1}{\hskip -2.8 pt}\hbox{\rm I}}

\newcommand{\la}{{\langle}}
\newcommand{\ra}{{\rangle}}

\newcommand{\infep}{{\infty,\ep}}

\newcommand{\Xit}{{X_{i,t}^N}}
\newcommand{\Xis}{{X_{i,s}^N}}

\newcommand{\Vit}{{V_{i,t}^N}}
\newcommand{\Vis}{{V_{i,s}^N}}

\newcommand{\Xjt}{{X_{j,t}^N}}

\newcommand{\Yit}{{Y_{i,t}^N}}
\newcommand{\Yis}{{Y_{i,s}^N}}

\newcommand{\Wit}{{W_{i,t}^N}}
\newcommand{\Wis}{{W_{i,s}^N}}

\newcommand{\Yjt}{{Y_{j,t}^N}}
\newcommand{\Yjs}{{Y_{j,s}^N}}

\newcommand{\muNXt}{{\mu^N_{X,t}}}

\newcommand{\cqfd}{{\unskip\kern 6pt\penalty 500
\raise -2pt\hbox{\vrule\vbox to 6pt{\hrule width 6pt
\vfill\hrule}\vrule}\par}}

\title{Propagation of chaos for the Vlasov-Poisson-Fokker-Planck system in 1D}
 
\author{Maxime Hauray and Samir Salem\footnote{
Universit{\'e} d'Aix-Marseille, CNRS, \'Ecole Centrale,  I2M, UMR 7373,  13453 Marseille, France.
\texttt{maxime.hauray@univ-amu.fr}, \texttt{samir.salem@univ-amu.fr}}
}

\date{}

\begin{document}
\maketitle

\begin{abstract}
We consider a  particle system in 1D, interacting via  repulsive or attractive Coulomb forces. We prove the trajectorial propagation of molecular chaos towards a nonlinear SDE associated to the Vlasov-Poisson-Fokker-Planck equation. We obtain a quantitative estimate of convergence in expectation, with an optimal convergence rate of order $N^{-1/2}$. We also prove some exponential concentration inequalities of the associated empirical measures.
A key argument is a weak-strong stability estimate on the (nonlinear) VPFP equation, that we are able to adapt for the particle system in some sense. 
\end{abstract}

{\bf Keywords:} Particles system, 1D Vlasov-Poisson equation, Propagation of molecular chaos, 
Monge-Kantorovich-Wasserstein distances, exponential concentration inequalities.

\medskip
{\bf AMS Subject Classification:} 35Q83 Vlasov-like equations,  60K35 Interacting random processes, 60H10   	Stochastic ordinary differential equations, 82C21   	Dynamic continuum models.

\section{Introduction}
We consider here a one dimensional system of $N$ particles, with position $X_i^N \in \R$ and velocity $V_i^N \in \R$, interacting via the Poisson interaction, and submitted to independent Brownian noises and friction. The associated system of Stochastic Differential Equation (SDE) is the following:
\begin{equation} \label{eq:Nps}
d\Xit = \Vit \,dt ,\qquad
d\Vit = \Bigl( \frac1N \sum_{j=1}^N K \bigl(\Xit - \Xjt \bigr) -  \Vit  \Bigr)\,dt + \sqrt 2  \,dB_{i,t},
\end{equation}
where the $(B_{i,t})_{t \ge 0}$ are independent Brownian motions. 
The interaction kernel is defined (everywhere) by 
\begin{equation} \label{eq:defK}
K(x) := \pm \frac12 \sign(x)  = \pm \frac12 
\begin{cases}
1  & \text{if } x >0, \\
0 & \text{if } x=0, \\
-1 & \text{if } x<0.
\end{cases}
\end{equation}

The case $K= \frac12 \sign$ corresponds to the repulsive case, while the interaction is attractive  when $K = - \frac12 \sign$. This two cases lead of course to very different dynamics, but  concerning the propagation of chaos in finite time the sign of $K$ is not very relevant, so we will handle both cases in the same way.  

\paragraph{Well-posedness of the particle system.}

The first non-obvious problem is raised by the particle system~\eqref{eq:Nps}. Since the force field is only of bounded variation ($BV$ in short) and not Lipschitz, the standard theory does not apply. Neither does the theory of existence and uniqueness developed for uniformly elliptic diffusions (see for instance\cite{Ver81,Davie,CatGub} among many others) since the diffusion act here only on the velocities.
However, due to the particular geometry of the problem, we can still get weak existence and uniqueness in law. 
The precise result is the following

\begin{thm} \label{thm:WPPS}
For any $N \ge 2$, and any (deterministic) initial condition $(X_{i,0}^N,V^N_{i,0})_{i \le N} \in \R^{2N}$, weak existence and uniqueness in law hold for SDE~\eqref{eq:Nps}.
\end{thm}

Since \eqref{eq:Nps} is linear, it also implies weak existence and uniqueness for any random initial condition. 	
Theorem~\ref{thm:WPPS} is proved in Section~\ref{sec:WPPS} using the following strategy:  we reformulate~\eqref{eq:Nps} in an SDE  with memory involving the
$(V^N_{i,t})_{i \le N, t\ge 0}$ only (simply because  the $X_{i,t}^N$ are time integrals of the $V^N_{i,t}$). Then we apply to that new (non-markovian) SDE a standard technique relying on the Girsanov's theorem. To the best of our knowledge, the strong existence and uniqueness of solution to that system are yet unknown, and we were not able to prove it.

\paragraph{A non-linear SDE related to the Vlasov-Poisson-Fokker-Planck equation.}

When the number of particles is large, and under the assumption that two particles picked up among all the others are roughly independent at any time (in particular this should be true at $t=0$), we expect the $N$ particles to behave almost like $N$ i.i.d copies of the (expected unique) solution to the following non linear SDE, or McKean-Vlasov process:
\begin{equation} \label{eq:NLSDE}
d Y_t = W_t \,dt, \qquad
d W_t = \E_{\bar Y_t} \bigl[ K(Y_t - \bar Y_t) \bigr] \,dt  - W_t \, dt +   \sqrt 2\, dB_t,
\end{equation}
where $(B_t)_{t  \ge 0}$ is a Brownian motion (independent of the rest) and $\bar Y_t$ is an independent copy of $Y_t$. 

We will prove the well-posedness of that nonlinear SDE for initial data with a uniform control on the velocity tails, and even a weak-strong stability estimate (here ``strong solution'' means that the law $\LL(Y_t)$ of $Y_t$ remains uniformly bounded in time). Before stating the precise result,  we introduce two useful norms:

\begin{defn}
For any $\lambda >0$, and any $\gamma >0$, we define  for any $ f \in L^\infty(\R^2)$ the  two following norms:
\begin{align}
\| f\|_{e,\lambda} &:=  \esup_{(x,v) \in \R^2} f(x,v) e^{\lambda | v |}  \in [0, + \infty] \label{normlambda} \\
\| f\|_{p,\gamma} &:=  \esup_{(x,v) \in \R^2} f(x,v) \la v \ra^{-  \gamma}
  \in [0, + \infty] \label{normgammaa} 
\end{align}
where $\la v \ra =\sqrt{1+v^2}$, so that equivalently $\| f\|_{p,\gamma} =  \esup_{(x,v) \in \R^2} f(x,v)\bigl(1+v^2\bigr)^{-  \gamma/2}$, and the essential supremum are taken with respect to the Lebesgue measure.
\end{defn}

\begin{thm} \label{thm:WPproc}

(i) Existence of strong solutions:

Assume that the law $f_0$ of the initial condition $(Y_0,W_0)$ satisfies $f_0 \in \PPP_1(\R^2) \cap  L^1(\R^2)$ and also  $\| f_0\|_{e,\lambda} < + \infty$, for some $ \lambda >0$ or $\| f_0\|_{p,\gamma} < + \infty$, for some $ \gamma >1$. Then, given any Brownian motion $(B_t)_{t \ge 0}$  (independent of the initial conditions) there exists a solution $(Y_t,W_t)$ to~\eqref{eq:NLSDE} with initial condition $(Y_0,W_0)$, and its law at time $t \ge 0$: $f_t = \LL(Y_t,W_t)$ satisfies either
\begin{equation} 
\label{eq:expdec}
\|  f_t \|_{e,\lambda}  \leq  2 \, e^{t+ \lambda +  \frac{\lambda^2}2}    \|  f_0 \|_{e,\lambda e^{-t}},
\quad \text{or} \quad
\| f_t \|_{p, \gamma} \le C_{\gamma} e^{\gamma t}  \|  f_0 \|_{p,\gamma},
\end{equation}
where $C_{\gamma}$ is a constant depending only on $\gamma$.

\medskip
(ii) Weak/strong stability for solution with bounded density (in $y$). 

If $(Y_t^1,W_t^1)$ and $(Y_t^2,W_t^2)$ are two solutions to~\eqref{eq:NLSDE} built on the same probability space with the same Brownian motion $(B_t)_{t \in \R}$,  if the density  $\rho^1_t = \LL(Y^1_t)$ is uniformly bounded at any time: $\| \rho_t \|_\infty <+\infty$ for any $ t \ge 0$, then the following stability estimate holds
\begin{equation} \label{eq:WSstab}
\E \bigl[ |Y^1_t - Y^2_t | + |W^1_t - W^2_t| \bigr] \le 
e^{8 \bigl(t + \int_0^t \| \rho^1_s\|_\infty \,ds \bigr)}
\E \bigl[ |Y^1_0- Y^2_0 | + |W^1_0 - W^2_0| \bigr].
\end{equation}
\end{thm}

The proof of the weak-strong stability estimate relies on the crucial Lemma~\ref{lem:rope} that allows to control the singularity of the force, when comparing the evolution of two solutions, assuming only that one of them has a bounded density in position. That Lemma was already used in \cite{Hau-X}, to get similar results for the associated deterministic particle system: \emph{i.e.\  }particles interacting via the same kernel but without noise and friction.
The proof of the existence of solutions relies on an usual approximation procedure, see Section~\ref{sec:FK}. The propagation of the bound on $\|f_t \|_{e,\lambda}$ or $\| f_t\|_{p,\gamma}$ is done using a standard argument that we may call the method of characteristics or Feynman-Kac's formula. 

\paragraph{The stability result on the Vlasov-Poisson-Fokker-Planck equation}

The stability results on the process~\eqref{eq:NLSDE} simply translate on the associated Fokker-Planck or Kolmogorov forward equation, which is here the Vlasov-Poisson-Fokker-Planck equation (VPFP in short):
\begin{equation} \label{eq:VPFP}
\partial_t f_t + v \, \partial_x f_t +  (\rho_t \star K) \partial_v f_t = \partial_v ( \partial_v f_t + v \, f_t ),
\end{equation}
where $f_t$ is the law at time $t$ of the process $(Y_t,W_t)$ and $\rho_t = \int f_t \, dv$ is the law at time $t$ of $Y_t$. As the kernel $K$ is bounded and defined everywhere by~\eqref{eq:defK}, remark that very few hypothesis are required  to define solutions to~\eqref{eq:VPFP} in the sense of distribution:  if $f  \in L^1_{loc}(\R^+,\PPP(\R^2))$, where $\PPP(\R^2)$ stand for the space of probability, then all the terms appearing in~\eqref{eq:VPFP} define a distribution.

The Theorem~\ref{thm:WPproc} as the following consequence on VPFP:

\begin{cor}[of Theorem~\ref{thm:WPproc}]  \label{prop:WPlaw}

(i) Existence of strong solution. 
Let $f_0 \in \PPP \cap L^1(\R^2)$  with a finite order one moment: $\int \bigl( |x|+ |v|\bigr) f_0 (dx,dv)$ , and satisfying either $\| f_0\|_{e,\lambda} < + \infty$, for some $ \lambda >0$ or $\| f_0\|_{p,\gamma} < + \infty$, for some $ \gamma >1$. Then, there exists a solution $f_t$ to~\eqref{eq:VPFP} with initial condition $f_0$, and it satisfies~\eqref{eq:expdec}.

\medskip
(ii) Weak/strong uniqueness for solution with bounded density (in $y$). 

If $f_t^1$ and $f_t^2$ are two solutions to~\eqref{eq:VPFP} and if $\rho^1_t = \int f^1_t \,dv$ is uniformly bounded for any time $t \ge 0$, then the following stability estimate holds:
\begin{equation} \label{eq:WSstab2}
W_1\bigl( f^1_t,f^2_t \bigr) \le 
e^{8 \bigl(t + \int_0^t \| \rho^1_s\|_\infty \,ds \bigr)}
W_1\bigl( f^1_0,f^2_0 \bigr)
\end{equation}
\end{cor}

This corollary is proved in Section~\ref{sec:WPlaw}. It is a direct consequence of Theorem~\ref{thm:WPproc}, and of the fact that  a weak solution $f_t$ to the VPFP1D equation~\eqref{eq:VPFP}, can always be represented as the time marginals of a process $(Y_t,W_t)_{t \ge 0}$ solution to~\eqref{eq:NLSDE}. 

\paragraph{The quantitative propagation of chaos  in the mean}
When comparing a solution $(\Xit, \Vit)_{i \le N}$ of the particle system~\eqref{eq:Nps} with the limit process~\eqref{eq:NLSDE} and its associated Fokker-Planck equation~\eqref{eq:VPFP}, a very natural strategy (which goes back to McKean~\cite{McKean}) is to introduce $N$ independent copies $(\Yit, \Wit)_{i \le N}$ of the limit nonlinear SDE~\eqref{eq:NLSDE}, constructed with the same Brownian motion as the $(\Xit,\Vit)$, and with initial conditions coupled in a optimal way.  In our case, we are able to prove a sharp estimate on the average distance between these two systems.

To state our result properly, we recall that by definition exchangeable random variables have  a law that is invariant  under permutation, and that chaotic sequences of r.v. are defined as follows:
\begin{defn}
Let $f$ be a probability on $\R^2$.
A sequence $\bigl((X_i^N,V^N_i)_{i \le N} \bigr)_{N \in \N}$ of exchangeable random variables is said to be $f$-chaotic, if one of the equivalent conditions below is satisfied:
\begin{itemize}
\item[i)] $\forall k \in \N, \quad \LL\bigl( (X_i^N,V^N_i)_{i \le k} \bigr) \xrightarrow[N \to \infty]{w} f^{\otimes k}$,
\item[ii)] $\ \LL\bigl( (X_1^N,V^N_1), (X_2^N,V^N_2) \bigr) \xrightarrow[N \to \infty]{w} f \otimes f$, 
\item[iii)] $ \ds  \mu^N := \frac1N \sum_{i=1}^N \delta_{(X^N_i,V^N_i)} \xrightarrow[N \to \infty]{\LL} f$.
\end{itemize}
\end{defn}
If $Y$ is a random variable of law $f$, we will equivalenty say that a sequence is $f$-chaotic or $Y$-chaotic. We refer to~\cite{sznitman} for the equivalence of the three conditions above, and to~\cite{HauMisch} for a quantitative version of that equivalence.

We also state the following proposition, that reformulate propagation of chaos in term of coupling. It is a consequence of~\cite[Theorem 1.2]{HauMisch} 
\begin{prop} \label{PoC-coup}
Assume that $\bigl((X_i^N,V^N_i)_{i \le N} \bigr)_{N \in \N}$ is a sequence of exchangeable random variables with uniformly bounded order two moment: $\sup_{N \in \N} \E \bigl[ |X^N_1|^2 + |V^N_1|^2 \bigr] < + \infty$. Let $f$ be a probability on $\R^2$. Then, the following statements are equivalent:
\begin{itemize}
\item[i)] The sequence $\bigl((X_i^N,V^N_i)_{i \le N} \bigr)_{N \in \N}$ is $f$-chaotic;
\item[ii)]
$ \ds
\E \Bigl[  W_1 \bigl( \mu^N, f \bigr)\Bigr] \xrightarrow[N \to \infty]{} 0
$, where $W_1$ stands for the order one Monge-Kantorovich-Wasserstein distance, and $\mu^N= \frac1N \sum_{i=1^N} \delta_{(X^N_i,V^N_i)}$ is the associated  empirical measure;
\item[iii)] If the $(Y_i,W_i)_{i \in \N}$ are i.i.d.r.v. with commun law $f$, independent of the  $(X_i^N,V^N_i)_{i \le N}$ for all $N$, then
\[
\E \Bigl[   \bigl|X_1^N-Y_1 \bigr|+\bigl|V_1^N-W_1 \bigr|   \Bigr] = 
\E \Bigl[  \frac1N \sum_i \bigl|X_i^N-Y_i \bigr|+ \bigl|V_i^N-W_i \bigr|\Bigr]
\xrightarrow[N \to \infty]{} 0;
\]
\item[iv)]
$ \ds  \min_{\text{coupling}}  
\E \Bigl[  \frac1N \sum_i \bigl|X_i^N-Y_i \bigr|+ \bigl|V_i^N-W_i \bigr|\Bigr]
\xrightarrow[N \to \infty]{} 0,
$
where the minimum is taken on all the exchangeable coupling with $(Y_i,V_i)$ i.i.d. with common law $f$.
\end{itemize}
\end{prop}

The above Proposition allows to state a precise result of propagation of molecular chaos. We emphasize that this result is a true result of propagation: it does not apply only to i.i.d. initial conditions, but to any chaotic initial conditions (with finite second order moment). However, the general case is somewhat more technical, so we warn the reader that the following theorem is simpler to understand if we only consider the case of i.i.d. initial conditions.

\begin{thm}\label{thm:MeaPC}
Let $f_0 \in \PPP(\R^2)$ with finite order two moment: $\int \bigl( |x|^2+ |v|^2\bigr)\,f_0(dx,dv) < \infty$, and such that there exists a (necessary unique by Corollary~\ref{prop:WPlaw}) solution $f_t$  to~\eqref{eq:VPFP} with initial condition $f_0$ satisfying $\int_0^t \| \rho_s\|_\infty \,ds < + \infty$ for any time $t \ge 0$, where $\rho_s$ stands for the density in position: $\rho_s(x) := \int f_s(x,dv)$. We also denote by $(Y_t,W_t)_{t \ge 0}$ the unique solution to~\eqref{eq:NLSDE} such that $\LL(Y_0,W_0) = f_0$. 

Let $(X_{i,0}^N,V_{i,0}^N)_{i \le N}$ be a sequence of $f_0$-chaotic random variable with uniformly bounded order two moment: $\sup_{N \in \N} \E \bigl[ |X^N_{1,0}|^2 + |V^N_{1,0}|^2 \bigr] < + \infty$.
By Theorem~\ref{thm:WPPS}, we may find a probability space together with  a $N$-dimensional  Brownian motion and a process $(\Xit,\Vit)_{i \le N}$ solution to~\eqref{eq:Nps}. Then, the sequence $\bigl((\Xit,\Vit)_{i \le N}\bigr)_{n \in \N}$ is $(Y_t,W_t)_{t \ge 0}$ chaotic.

More precisely,  by standard arguments, we may also construct on that probability space $N$ i.i.d. copies $(\Yit,\Wit)_{i \le N}$ of the solutions of~\eqref{eq:NLSDE} with the same Brownian motion, and with initial conditions of law $f_0$, coupled with $(X_{i,0}^N,V_{i,0}^N)_{i \le N}$ in an exchangeable way.
Then the following estimate holds:
\[
\E \, \biggl [  \, \sup_{s\in [0,t]} \bigl|X_{1,s}^N-Y_{1,s}^N \bigr|+ \bigl|V_{1,s}^N-W_{1,s}^N \bigr| \,  \biggr ] \, \leq \,   
\biggl( \E \Bigl[   \bigl|X_{1,0}^N-Y_{1,0}^N \bigr|+\bigl|V_{1,0}^N-W_{1,0}^N \bigr|   \Bigr] + 
\frac9{\sqrt N } \biggr)\, 
e^{\bigl(t+8\int_0^t \| \rho_s \|_{\infty} \, ds\bigr)}.
\]
In particular, if $(X_{i,0}^N,V_{i,0}^N)_{i \le N}$ are i.i.d. with law $f_0$, then 
\[
\E \, \biggl [  \, \sup_{s\in [0,t]} \bigl|X_{1,s}^N-Y_{1,s}^N \bigr|+ \bigl|V_{1,s}^N-W_{1,s}^N \bigr| \,  \biggr ] \, \leq \,   
\frac9{\sqrt N } 
e^{\bigl(t+8\int_0^t \| \rho_s \|_{\infty} \, ds\bigr)}.
\]
\end{thm}

\begin{rem} Thanks to Corollary~\ref{prop:WPlaw}, the last hypothesis on $f_0$ is satisfied if $\| f_0\|_{e,\lambda} < \infty$ for some $\lambda >0$, or if 
$\| f_0\|_{p,\gamma} < \infty$ for some $\gamma >1$.
\end{rem}
The proof of that result is performed in Section~\ref{sec:MeaPC}: it relies in a crucial way on a ``rope argument'' introduced in Section~\ref{sec:WSSU}. A second key argument is the introduction of an ad-hoc Poisson Random Measure (PRM), which allows to conclude using standard properties of PRMs.

Using standard results on the convergence on empirical measures toward their mean~\cite[Theorem 1]{FG}, we may deduce convergence results between the empirical measure of the particle system and the expected limit profile.
\begin{cor}[of Theorem~\ref{thm:MeaPC}] \label{cor:mean}
Under the same assumptions than in Theorem~\ref{thm:MeaPC}, and assuming moreover that $\int \bigl( |x|+ |v|\bigr)^q \,f_0(dx,dv) < \infty$ for some $q>2$ we obtain that for any time $t \ge 0$,
\[
\E \bigl[   W_1(\mu^N_t, f_t ) \bigr]  \le C_t \biggl( \frac{\ln(1+N)}{\sqrt N} + \E \bigl[   W_1(\mu^N_0, f_0) \bigr] \biggr),
\]
for some constant $C_t>0$ depending on $t$,$q$ and $f_0$.
\end{cor}

\paragraph{Exponential concentration inequalities for the particle system.}

In the case were the initial conditions are  i.i.d., we also prove concentration inequalities for the solutions of the particle system~\eqref{eq:Nps}, precisely:
\begin{thm} \label{thm:main}
Let $f_0 \in \PPP(\R^2)$ with some finite exponential moment: $\int  e^{\lambda( |x|+ |v|)} \,f_0(dx,dv) <\infty$ for some $\lambda >0$, and such that there exists a (necessary unique by Corollary~\ref{prop:WPlaw}) solution $f_t$  to~\eqref{eq:VPFP} with initial condition $f_0$ satisfying $\kappa_t:=\sup_{s\leq t} \| \rho_s\|_\infty \,ds < + \infty$ for any time $t \ge 0$, where $\rho_s$ stands for the density in position: $\rho_s(x) := \int f_s(x,dv)$. 
Let $(X_{i,0}^N,V_{i,0}^N)_{i \le N}$ be a sequence of $N$  i.i.d random variables with common law $f_0$.

By Theorem~\ref{thm:WPPS}, we may find a probability space together with  a $N$-dimensional  Brownian motion and a process $(\Xit,\Vit)_{i \le N}$ solution to~\eqref{eq:Nps}. By standard arguments, we may also construct on that probability space $N$ copies $(\Yit,\Wit)_{i \le N}$ of the solutions of~\eqref{eq:NLSDE} with the same initial conditions $(X_{i,0}^N,V_{i,0}^N)_{i \le N}$ and Brownian motion. 

Then the following concentration inequality holds for $  \lambda N^{-1/2} \le \ep \le  (5\kappa_t\wedge 1)\min\bigl( \frac1{16} ,\frac{\lambda}{2} ,\lambda^{-2}\bigr)$:
\[
\Prob \biggl( \, \frac1N \sum_{i=1}^N  \sup_{s \in [0,t]}  \bigl( |\Xis - \Yis| + |\Vis - \Wis| \bigr)   \ge  \bs{B_t} \ep \biggr) \le 
(t + \ep) \bigl( \bs{A_t} + \bs{A_t'} \sqrt N \ep \bigr) N^{\frac32} e^{- 2 N \ep ^2}, 
\]
where the three constants depend on $t$, $\lambda$ and the initial conditions $f_0$. See the end of Section~\ref{sec:main} for precise values.
\end{thm}

\begin{rem}
Thanks to Corollary~\ref{prop:WPlaw}, the hypothesis on $f_0$ are satisfied  if $\| f_0 \|_{e,\lambda'} < \infty$ for some $\lambda' > \lambda$. 
\end{rem}

The proof of Theorem~\ref{thm:main} relies on a different technique than the one used in the proof of Theorem~\ref{thm:MeaPC}. Here, we rather use exponential concentration inequalities on discrete infinite norms of empirical measures, and on some fluctuation terms appearing naturally when comparing solutions to~\eqref{eq:Nps} to copies of solutions to the nonlinear SDE~\eqref{eq:NLSDE}.

\medskip

Using deviation upper bounds  for the approximation of probability measure by random empirical measures associated to i.i.d sample, as for instance in~\cite[Theorem 2]{FG}, we can also obtain the following corollary.
\begin{cor} \label{cor:prop}
Under the same assumption as in Theorem~\ref{thm:main}, and if moreover the initial positions and velocities $(X_{i,0}^,V_{i,0}^N)$ are i.i.d random variables with law $f_0$, then 
for any $ T \ge 0$ and any  $\lambda >0$ there exists two constants $C_1, \; C_2$ such that for any $\ep  \ge 0$ satisfying  $  \lambda N^{-12} \le \ep \le (5\kappa_t\wedge 1) \min\bigl( \frac1{16} ,\lambda ,\lambda^{-2}\bigr)$, we have:
\[
\sup_{t \in [0,T]}  \Prob \Bigl[  \;   W_1( \muNXt, f_t)   \ge  \ep \Bigr] \le
C_1  \Bigl( \,N^4 e^{- C_2 N  \ep^2} + e^{-C_2 (N \ep)^{1 - \alpha}} + e^{-C_2 N \ep^2/ (1 - \ln \ep)^2} \Bigr),
\]
where $f_t$ is the unique solution of the VPFP equation~\eqref{eq:VPFP} with initial condition $f_0$.
\end{cor}

\paragraph{Exponential concentration for discrete infinite norms.}

One of the key ingredient of the proof of Theorem~\ref{thm:main} is a deviation inequality on time integral (or supremum) of discrete infinite norms for the $\mu^N_{Y,t} := N^{-1} \sum_i \delta_{Y^{N,i}_t}$, where the $Y^{N,i}_t$ are $N$ i.i.d. r.v. solutions to a given SDE. Such a result have a interest by itself so we state it below

\begin{prop} \label{prop:DevBound}
Let $(Y_i,W_i)_{i _le N}$ be $N$ i.i.d copies of a solution of~\eqref{eq:NLSDE} driven by independent Brownian motions $(B_{i,t})_{i \le N, t \ge 0}$, and assume that for $t \ge 0$:
\begin{itemize}
\item the common initial condition has a law $f_0$ which admits a finite exponential moment (in position and velocity) for some $\lambda >0$: $\E\Bigl[ e^{\lambda(|Y_0| + |W_0|)} \Bigr] < + \infty$. In particular we denote $c_\lambda := \frac52 + \frac1\lambda \ln \E \bigl[ e^{\lambda  |W_0|} \bigr]$;
\item $ \kappa_t := \sup_{0 \le s \le t} \| \rho_s \|_\infty < +\infty$ where $\rho_s$ stands for the time marginal of $Y_i$ at time $s$.
\end{itemize}
Then  provided that $\lambda N^{-1/2} \le \ep \gamma \le 5 \kappa_t \min\Bigl( \frac1{16}, \lambda^{-2} \Bigr)$, the following bound holds with $\bs{C_t}  := 10 +  \kappa_t^2 \lambda^{-1}  
 +  e^{\lambda(1/2+ \lambda)t}\,\E\Bigl[ e^{\lambda(|Y_0| + |W_0|)} \Bigr]$:
\[
\Prob \biggl(  \sup_{s \in [0,t]} \bigl\| \rho^N_s  \bigr\|_\infep \ge \kappa_t + \gamma \biggr) 
 \le \bs{C_t} \biggl( 1 + t \frac{2 \kappa_t + \gamma}\lambda  \bigl(  c_\lambda + \sqrt N  \ep \gamma  \bigr)   \biggr)\,N^{\frac32} \,   e^{- 2 N(\ep \gamma)^2},
\]
\end{prop}

Essentially, the bound behave like $N^{3/2} e^{-2N (\ep \gamma)^2}$, in the most interesting range, when $\ep \gamma \sim N^{-1/2}$.

\paragraph{Some related works}
The literature on the convergence of particle systems towards non linear mean-field models is quite huge, so we will restrict ourselves to one dimensional models. The usual strategy, valid for smooth interaction, is well explained in Lecture notes by Sznitman~\cite{sznitman}.  In~\cite{Cepa} C\'epa and L\'epingle prove the propagation of chaos for the Dyson model: an order one model (\emph{i.e.\ }without velocities) with a strongly singular interaction $K(x) \sim |x|^{-1}$ modeling the behavior of eigenvalues of large hermitian matrices.  Their proof relies on the use of maximal monotone operators. Recently, that convergence result was extended to similar systems with even stronger interaction $K(x) \sim |x|^{-1-\alpha}$ with $\alpha \in [0,1)$, by Berman and \"Onnheim~\cite{Berman} using Wasserstein gradient flows. 
For second oder models (involving positions and velocities), our result is to the best of our knowledge the first result of propagation of chaos with the Poisson singularity in the stochastic case. In the deterministic case, \emph{i.e.\ }when the system under study is~\eqref{eq:Nps} without the Brownian motions, then the mean field limit was proved originally by~\cite{Trocheris}, and then by~\cite{CGP-ARMA} as a special case  of semi-geostrophic equations, and again by the first author~\cite{Hau-X}. We shall finally mention the very recent preprint of Jabin and Wang~\cite{JW}, which proves the propagation of chaos in a very similar setting.  They consider system like~\eqref{eq:Nps}, with or without noise, with a very weak assumption on the interaction: $K$ bounded, but a strong regularity assumption on the limit, and prove some quantitative propagation of chaos in term of relative entropy. 

\paragraph{Plan of the paper.}

The paper is organized in the following way. In section~\ref{sec:WPPS}, we prove Theorem~\ref{thm:WPPS} (weak existence and uniqueness of the particle system). In Section~\ref{sec:WPproc}, we focus on the nonlinear limit SDE and prove Theorem~\ref{thm:WPproc} and its corollary~\ref{prop:WPlaw}, and also a useful proposition about propagation of moments (Proposition~\ref{prop:MomDev}). Section~\ref{sec:MeaPC} is devoted to the proof of Theorem~\ref{thm:MeaPC} on the propagation of chaos in the mean, Section~\ref{sec:DevBound} to the proof of Proposition~\ref{prop:DevBound} and Section~\ref{sec:main} to the proof of Theorem~\ref{thm:main} on the exponential concentration. A useful regularity lemma is proved in the Appendix.

%
%
%
%
%
%
%
%
%
\section{Proof of Theorem~\ref{thm:WPPS}} \label{sec:WPPS}

\paragraph{\sl A related SDE with memory.}
First note that a weak solution to equation ~\eqref{eq:Nps} is 
some stochastic basis, together with a $N$ dimensional Brownian motion $(B_t^N)_{t \ge 0}$  on it and a $\R^{2N}$ valued processes $(X^N_t,V^N_t)_{t \ge 0}$  satisfying  for all $t \ge 0$
\[
X_t^N=X_0^N+\int_0^t V_s^N \,ds,
\qquad 
V_t^N=V_0^N+\int_0^t K^N(V_s^N)ds-\int_0^t V_s^N \,ds+\sqrt 2 \, B_t^N	,
\]
where we denoted $K^N(x_1,\cdots,x_N)$ the vector valued field which $i$-th component is $\frac{1}{N}\sum_{j\neq i}K(x_i-x_j)$. 
But, from that system, we may write a  SDE ``with memory'' involving $V^N$ only:
\begin{equation} \label{delayedSDE}
\forall \, t \ge 0, \quad 
V_t^N=V_0^N+\int_0^t K^N \biggl(X_0^N  + \int_0^s V^N_u \,du  \biggr) \,ds-\int_0^t V_s^N \,ds + \sqrt 2 \, B_t^N	.
\end{equation}

Conversely, given a solution to the SDE~\eqref{delayedSDE}, it is not difficult to construct a solution to the original system. So it will be enough to prove weak existence and uniqueness in law for the delayed SDE~\eqref{delayedSDE}.

\paragraph{\sl Weak \ existence.}	
Let $(\Omega, \mathcal{F},(\mathcal{F})_{t\ge 0}, \mathbb{P})$ be a stochastic  basis and  $(B^N_t)_{t \ge 0}$ be a $N$-dimensional on it. We define 
\[
U^N_t:=-   \frac1{\sqrt 2} \int_0^t  K^N\Bigl(X_0^N+sV _0^N + \sqrt 2 \int_0^s B_u^Ndu\Bigr)ds 
+ \int_0^t \biggl( B_s^N+\frac1{\sqrt 2}V_0^N \biggr) ds + B_t^N,
\quad
V_t^N:=V_0^N+\sqrt 2 \, B^N_t. 
\]
The above definition of the two r.v. $(U^N_t, V^N_t)_{t \ge 0}$ implies that for any $t \ge 0$,
\[
V_t^N=V_0^N+\int_0^t K^N \biggl(X_0^N  + \int_0^s V^N_u \,du  \biggr) \,ds-\int_0^t V_s^N \,ds + \sqrt 2 \, U_t^N	,
\]
which is exactly~\eqref{delayedSDE} with $B^N_t$ replaced by $U^N_t$. So, it remains to apply Cameron-Martin-Girsanov (CMG) theorem:  with an appropriate change of the reference probability measure, $(U_t^N)_{t \ge 0}$ can be considered has a $N$-dimensional Brownian motion. For this, remark that $d U^N_t =  - H_t^N \,dt + d B^N_t$, where 
\[
H_t^N:= \frac1{\sqrt 2}\, K^N\biggl(X_0^N+ tV_0^N+\int_0^t \sqrt 2 B_u^Ndu \biggr)
-   B_t^N - \frac1{\sqrt 2} V_0^N,
\]
is $\mathcal{F}_t$-adapted and progressively measurable, and that  for $0 <\gamma < \frac1{6t}$
\[
\E \Bigl[e^{\gamma |H_t^N|^2} \Bigr] \leq e^{ \frac32 \gamma \bigl(  N \| K  \|^2_\infty+|V_0^N|^2 \bigr)} 
\E \Bigr[e^{3\gamma(B_t^N)^2}\Bigl] < + \infty.
\]
 Therefore we deduce from classical results about exponential martingales that the process $Z^N_t$ defined by
\[
Z^N_t=\exp\left ( \int_0^t  H^N_s \cdot  dB^N_s  -\frac{1}{2}\int_0^t|H^N_s|^2ds\right )
\]
(where $\cdot$ stands for the scalar product) is a martingale, 
and due to CMG theorem $(U^N_t)_{t \ge 0}$ is a $N$-dimensional Brownian motion under the probability $\mathbb{Q}$ defined for any $A$ in $\mathcal F_t$ by
$
\mathbb{Q}(A)=\int_{A}Z^N_t \, d\Prob
$.
Therefore $\bigl(\Omega,  \mathcal{F} ,(\mathcal{F}_t)_{t\ge 0}, \mathbb{Q}, (U^N_t,V^N_t)_{t \ge 0} \bigr)$ is a weak solution to SDE ~\eqref{delayedSDE}.

\paragraph{\sl Uniqueness \ in \ law.}
Suppose that $(\Omega,\mathcal{F},(\mathcal{F})_{t \ge 0},\mathbb{P},(V^N_t,B^N_t)_{t \ge 0}$  is a solution to equation~\eqref{delayedSDE} with initial condition $V_0^N$. We may use CMG theorem again: in fact $2^{-1/2} dV^N_t  = - \tilde H_t^N \,dt + dB^N_t$, where
\[
\tilde H^N_t := - \frac1{\sqrt 2} K^N\biggl(X^N_0+\int_0^t V_u^N \,du\biggr) +  \frac1{\sqrt 2} V_t^N.
\]
Moreover by~\eqref{delayedSDE}, $V^N_t$ also satisfies
\[
V^N_t = e^{-t} V^N_0 + \int_0^t  e^{s-t} K^N\biggl(X^N_0+\int_0^s V_u^N \,du\biggr) ds + \sqrt 2 \int_0^t e^{s-t} dB^N_s,
\]
which implies that 
\[
\bigl|\tilde H^N_t \bigr| \le \bigl|V^N_t \bigr|+ \frac{\sqrt N}2 \le \bigl|V^N_0 \bigr| + \frac{3\sqrt N}2 + \bigl|M^N_t \bigr|,
\] 
where $M^N_t$ is a Gaussian r.v.  with law $\mathcal N \bigl(0, (1-e^{2t}) \mathrm{Id} \bigr)$.  It follows that $\E \bigl[ e^{\gamma |\tilde H^N_t |^2}\Bigr] < \infty$, for $t \ge 0$ and $\gamma < 1/2$. This implies that the process $\tilde Z^N_t$ defined by
\[
\tilde Z^N_t=
\exp\left ( \int_0^t   \tilde H^N_s \cdot dB^N_s  -\frac{1}{2}\int_0^t \bigl| \tilde H^N_s \bigr|^2 \, ds \right )
= \exp\left ( \frac1{\sqrt 2} \int_0^t   \tilde H^N_s \cdot dV^N_s  + \frac{1}{2} \int_0^t \bigl| \tilde H^N_s \bigr|^2 \, ds \right )
,
\]
is a martingale, and by Cameron-Martin-Girsanov theorem,  $2^{-1/2} (V^N_t-V^N_0)_{t\geq 0}$ is a $N$-dimensional Brownian motion on the filtered space  $(\Omega, \FF,(\FF)_{t\ge0}, \tilde{\Q})$ where $\tilde \Q$ is defined for any $A$ in $\mathcal F_t$ by $\tilde \Q(A)=\int_{A} \tilde Z^N_t \, d\Prob$.

Now for any ``cylindrical'' function $\phi$ on $C(\R^+,\R^N)$ of the form $\phi\bigl((V^N_s)_{s \ge 0} \bigr)= \varphi_1(V^N_{t_1}) \times \cdots  \times \varphi(V^N_{t_k})$, we get for $t \ge t_k$:
\begin{align*}
\E_\Prob \Bigl[  \phi\bigl( (V^N_s)_{s \ge 0} \bigr)\Bigr]  & =
\E_{\tilde \Q}  \Bigl[  \phi\bigl( (V^N_s)_{s \ge 0} \bigr) \, \bigl(\tilde Z^N_t \bigr)^{-1} \Bigr] \\
& =
\E_{\tilde \Q}  \biggl[  \phi\bigl( (V^N_s)_{s \ge 0} \bigr)
\exp\left ( -\frac1{\sqrt 2} \int_0^t   \tilde H^N_s \cdot dV^N_s  - \frac{1}{2} \int_0^t \bigl| \tilde H^N_s \bigr|^2 \, ds \right )
\biggr].
\end{align*} 
The expression on the last line does not involve $(B^N_t)_{t \ge 0}$ anymore. Since the law of $(V^N_t)_{t \ge 0}$ under $\tilde \Q$ is the law of the Brownian motion, and since $\tilde Z^N_t$ can be expressed in term of $V^N_t$ only, that last expression does not depend on the specific solution we selected at the beginning of this paragraph: we will obtain exactly the same formula starting from a second solution. Since, solution to~\eqref{delayedSDE} have continuous trajectories, this implies the uniqueness in law  
of the solutions to the SDE~\eqref{delayedSDE} and also to~\eqref{eq:Nps}.

%
%
%
%
%
%
%
\section{Proof of Theorem~\ref{thm:WPproc}} \label{sec:WPproc}


\subsection{Weak-strong stability and uniqueness.} \label{sec:WSSU}

\paragraph{A key bound.}
Here we will use a ``rope'' argument which has already been used in \cite{Hau-X} to treat the propagation of chaos for deterministic VP1D equation. It consists in noticing that:
\begin{equation*}
K(x- \bar x)-K(y - \bar y)=0  
\end{equation*}
as soon as $|y - \bar y|>|x-y|+|\bar x - \bar y|$. It is also interesting to replace the later condition by the stronger one $|y - \bar y|> 2 \max(|x-y|,|\bar x - \bar y|)$.
Since $K$ is bounded by $1/2$, it implies the following bound, that we will use many times in the sequel
\begin{equation} \label{eq:rope}
 \left|K(x- \bar x)-K(y - \bar y)\right| \le   \indiq_{|y - \bar y|\leq 2\,|x-y|}+ \indiq_{|y - \bar y|\leq 2\,|\bar x - \bar y|}
\end{equation}


That one-sided condition (the indicator functions in the r.h.s. take only $y-\bar y$ as argument) allows to prove a simple but important Lemma, where discrete infinite norms are introduced:
\begin{defn} \label{def:infep}
For any $\ep>0$, and any $f \in \PPP(\R)$, we define the infinite norm at scale $\ep$, denoted $\| f \|_\infep$ by:
\[
\| f \|_\infep := \sup_{ x \in \R} \frac{f\bigl([x- \ep, x+ \ep] \bigr)}{2 \ep} = \Bigl \|  f \star  \frac1{2\ep} \indiq_{[-\ep,\ep]} \Bigr\|_\infty.
\] 
\end{defn}
\begin{lem} \label{lem:rope}
Assume that $(X,Y)$ is a random couple of real numbers and that $(\bar X,\bar Y)$ is an independent copy of that couple. 

\medskip
(i) Then, 
\[
\E \bigl[  K(X-\bar X) - K(Y- \bar Y) \bigr] \le  8\, \min\bigl(  \| \rho_X\|_\infty , \| \rho_Y\|_\infty \bigr) \E \bigl[  |X-Y| \bigr],
\]
where $\rho_X,\rho_Y$ denote respectively  the density (with respect to the Lebesgue measure) of the law of $X$ and  $Y$.

\medskip
(ii) For $\ep>0$, we also have a similar estimate involving the discrete infinite norms  $\| \cdot\|_\infep$ defined in Definition~\ref{def:infep}:
\[
\E \bigl[  \bigl| K(X-\bar X) - K(Y- \bar Y)  \bigr| \bigr] \le  8\, \min\bigl(  \| \rho_X\|_\infep , \| \rho_Y\|_\infep \bigr) \Bigl( \E \bigl[  |X-Y| \bigr] + \frac \ep 2 \Bigr).
\]

\medskip
(iii) In the case where $(X,Y)$ and $(\bar X, \bar Y)$ are still independent but with possibly different distributions, we get the more general estimate: for $\ep, \bar \ep \ge 0$ (in  case set $\ep=0$ or $\bar \ep=0$, set $ \| \cdot  \|_{\infty,0} = \| \cdot \|_\infty$),
\[
\E \bigl[  \bigl| K(X-\bar X) - K(Y- \bar Y) \bigr|  \bigr] \le  4\,  \| \rho_{\bar Y}\|_{\infty,\bar \ep} \Bigl(\E \bigl[  |X-Y| \bigr]  + \bar \ep /2 \Bigr)+  
4\, \| \rho_Y\|_\infep   \Bigl( \E \bigl[  |\bar X- \bar Y| \bigr] + \ep/2 \Bigr).
\]

\end{lem}
\begin{proof} We first prove $i)$. 
Starting from~\eqref{eq:rope}, we may bound 
\begin{align*}
\E \bigl[   \bigl| K(X-\bar X) - K(Y- \bar Y)  \bigr| \bigr]
&  \le  \E \bigl[  \indiq_{|Y - \bar Y| \leq 2\,|X-Y|}+ \indiq_{|Y - \bar Y|\leq 2\,|\bar X - \bar Y|}  \bigr]  \\
& \le  2\, \E \bigl[  \indiq_{|Y - \bar Y| \leq 2\,|X-Y|}  \bigr]  \\
& \le 2 \, \E \Bigl[ \E \bigl[  \indiq_{|Y - \bar Y| \leq 2\,|X-Y|} \big| (X,Y) \bigr] \Bigr] \\
& \le 2 \, \E \Bigl[  4\, \| \rho_Y\|_\infty \,  |X-Y|  \Bigr] \\
& \le 8\, \| \rho_Y\|_\infty  \E \bigl[  |X-Y| \bigr].
\end{align*}
In the second line, we used that the expectation remain unchanged if we permute $(X,Y)$ with $(\bar X,\bar Y)$. In the fourth, we used that the law of $\bar Y$ has density $\rho_Y$. If we apply the previous calculation with the couple $(Y,X)$ and $(\bar Y, \bar X)$, we obtain a similar result, with $\| \rho_X \|_\infty$ in place of  $\| \rho_Y \|_\infty$.

To obtain $ii)$, remark that for any interval $[a,b] \subset \R$:
\[
\E \bigl[  \indiq_A (X) \bigr] = \int_a^b \rho_X(dx)  \le \| \rho_X \|_\infep \bigl( |b-a| + 2\ep \bigr).
\]
Indeed, to use discrete infinite norm, we need to cover $[a,b]$ by a union of small intervals of length $2\ep$. For this at most (the integer part of) $|b-a|/(2 \ep) + 1$ such intervals are requested.

For the third point, we cannot use the permutation $(X,Y) \leftrightarrow (\bar X,\bar Y)$ in the previous calculation and have to estimate the two terms separately. The necessary adaptations are straightforward.
\end{proof}

\paragraph{A simple Gr\"onwall lemma.}
That bound allows us to prove the weak-strong stability part of Theorem~\ref{thm:WPproc}. We introduce 
$(X_t,V_t)_{t \in \R_+}$ and $(Y_t,W_t)_{t \in \R_+}$ two solutions of the non-linear SDE~\eqref{eq:NLSDE} constructed on the same Brownian motion $(B_t)_{t \in \R}$, and  
also  $(\bar X_t,\bar V_t, \bar Y_t,\bar W_t)_{t \in \R}$ an independent copy of  the previous coupled processes.  We also assume that $\rho_t := \LL(X_t)$ has a bounded density for all $t \ge 0$. 
Then $(X_t-Y_t,V_t-W_t)_{t \in \R}$ solves the following ODE system:
\[
\frac d{dt} (X_t-Y_t) =  V_t-W_t, \qquad
\frac d{dt}   (V_t-W_t) =      -  (V_t-W_t) + \E_{(\bar X_t, \bar Y_t )} \bigl[  K(X_t - \bar X_t) - K(Y_t - \bar Y_t)  \bigr],
\]
which naturally leads to 
\begin{align*}
\sup_{s \in [0,t]} |  X_s-Y_s |   & \leq   |  X_0-Y_0 |  +  \int_0^t  | V_s-W_s|  \,ds,   \\
\sup_{s \in [0,t]} | V_s-W_s| 
& \le   | V_0-W_0| + \int_0^t   \bigl|K(X_s - \bar X_s) - K(Y_s - \bar Y_s) \bigr|  \,ds.
\end{align*}
If we take the expectation in the previous system, and apply the point (i) of Lemma~\ref{lem:rope} to the couple $(Y_t,X_t)$ and $(\bar Y_t, \bar X_t)$, we may write: 
\begin{align*}
\E \biggl[ \sup_{s\in [0,t]}  |  X_s-Y_s |  \biggr] 
& \leq \E [| X_0-Y_0|] +  \int_0^t  \E \bigl[ | V_s-W_s  | \bigr]  \, ds\\
\E \biggl[  \sup_{s\in [0,t]}  | V_s-W_s | \biggr]  
& \leq \E \bigl[| V_0-W_0| \bigr]  + 8 \int_0^t  \|  \rho_s \|_\infty   \E  \bigl[  | X_s- Y_s| \bigr]  \,ds.
\end{align*}

Summing up the two inequalities and applying the Gr\"onwall lemma lead to the requested estimate
\[
\E \biggl[ \sup_{s\in [0,t]} \bigl(  |  X_s-Y_s | + |V_s - W_s |   \bigr) \biggr]  
\le 
\E \bigl[ \  |  X_0-Y_0 | + |V_0 - W_0 |  \biggr]  \exp \biggl( t  + 8 \int_0^t \|  \rho_s\|_{\infty} \,ds \biggr).
\]

Remark that it is not completely straightforward to take advantage of the restoring force in order to improve the above bound, especially because of the supremum in time.

\subsection{Propagation of moments} \label{sec:DevMom}

Here, we show that order one moments, and exponential moments are propagated by the SDE~\eqref{eq:NLSDE}.
We emphasize that our results apply not only for $K$ given by~\eqref{eq:defK}, but as soon as $\| K \|_\infty \le 1/2$. In particular, we will apply it later to nonlinear SDE where $K$ is replaced by a smooth mollification.

\begin{prop} \label{prop:MomDev}

Let be $(Y_t,W_t)_{t\ge 0}$ be a weak solution to~\ref{eq:NLSDE}, with given (random) initial condition $(Y_0,W_0)$, and with an interaction kernel $K$ which is not necessary given by~\eqref{eq:defK} but satisfies $\| K \|_\infty \le \frac12$. If $(Y_0,W_0)$ has an exponential moment of order $\lambda>0$, it holds for $t \ge 0$:
\begin{align*}
(i) & \qquad  
\E \Bigl[e^{\lambda |W_t|} \Bigr] \leq 
   e^{ \frac \lambda 2(3+ \lambda)} \E \Bigl[e^{\lambda e^{-t} |W_0|} \Bigr],
\\
(ii) & \qquad 
\E \Bigl[ e^{\lambda |Y_t|} \Bigr]  \le  2\, e^{\lambda t \bigl( \frac 12 + \lambda \bigr)  } \,
 \E \Bigl[ e^{\lambda (|Y_0| + |W_0|)} \Bigr]:=2\, e^{\lambda t \bigl( \frac 12 + \lambda \bigr)  } \, \mathcal{M}^{x,v}_{\lambda}(f_0).
\end{align*}

It also holds for any $0 \le s < t \le s+ \min\Bigl( \frac 1 {16},\lambda^{-2} \Bigr)$:
\begin{align*}
(iii) \qquad 
\E \Bigl[ e^{\lambda (t-s)^{-1} \sup_{s \le u \le t} |Y_u - Y_s|}\Bigr]  \le
e^{\frac \lambda 2(5 + \lambda)} \E \Bigl[e^{\lambda |W_0|} \Bigr].
\end{align*}

Lastly, simpler estimates on the order one moments also hold: for any $0 \le s < t \le  s+ \frac14$,
\[
(iv) \quad
\E\bigl[ |W_t| \bigr] \le e^{-t} \, \E\bigl[ |W_0| \bigr]  +2,
\qquad
\E\bigl[ |Y_t -Y_s| \bigr] \le |s-t| \Bigl( \E\bigl[ |W_0| \bigr]  + 3  \Bigr).
\]
\end{prop}

\begin{proof}

{\sl Point $(i)$.}
First, introducing the notation $f_t$ for the time marginal of $(Y_t,W_t)$ and  $F(t,x) := \int K(x-y) f_t(dy,dw)$ we have by~\eqref{eq:NLSDE}
\begin{equation} \label{equal:W}
e^t W_t  = W_0 + \int_0^t e^s F(s,Y_s)\,ds + \sqrt 2 \int_0^t e^s dB_s.
\end{equation}
Next, since $|K|$ and also $|F|$ are bounded by $1/2$, we get a simple inequality 
\begin{equation}
 |W_t| \le e^{-t} |W_0| + \frac12  + |M_t|, \quad \text{with} \quad
M_t := \sqrt 2 \int_0^t e^{s-t} dB_s. \label{bound:W}
\end{equation}
$M_t$ is a centered Gaussian random variable with variance $1-e^{-2t} \le 1$, for which exponential moments 
are simple to obtain. In fact for a Gaussian variable $Z \sim \NN (0,\sigma^2)$ we have a simple bound
\[
\\E\left ( e^{\lambda  | Z |} \right ) =  \frac2{\sqrt{2\pi } \sigma}\int_0^{+\infty} e^{\lambda x -\frac{x^2}{2  \sigma^2}}\, dx
=  \frac 2 {\sqrt\pi }  e^{\frac{\lambda^2 \sigma^2}2} \int_{- \lambda \sigma \,  2^{-1/2}}^{+\infty} e^{-x^2  }\, dx 
 =  e^{\frac{\lambda^2 \sigma^2}2}  \biggl( 1 + \erf\biggl(\frac{\lambda \sigma}{\sqrt 2}\biggr) \biggr),
\]  
where we used for the error function $\erf$ the definition $\erf(x) := \frac2{\sqrt \pi} \int_0^x e^{-x^2} \,dx$. Using that $\erf(x) \le \min \bigl(1, e^x \bigr)$, we finally get the following bound, that will be very useful in the sequel:
\begin{equation} \label{expmomG}
\\E\left ( e^{\lambda  | Z |} \right ) \le 
\min \Bigl( 2, e^{\frac1{\sqrt 2} \lambda \sigma} \Bigr) e^{\frac{\lambda^2 \sigma^2}2} 
\end{equation}
Together with the independence of the Brownian motion $(B_t)_{t \ge 0}$ and the initial condition $W_0$, it leads here to
\begin{align*}
\E \Bigl[ e^{\lambda |W_t|}\Bigr] & \le 
e^{ \frac \lambda 2 } \E \Bigl[ e^{\lambda e^{-t} |W_0|}\Bigr] \E \Bigl[e^{\lambda |M_t|}\Bigr]
 \le 
e^{ \frac \lambda 2 } \E \Bigl[ e^{\lambda e^{-t} |W_0|}\Bigr] e^{\frac1{\sqrt 2} \lambda  + \frac12 \lambda^2}\\ 
&\le    e^{ \frac \lambda 2 (3+ \lambda) }  
 \E \Bigl[ e^{\lambda e^{-t} |W_0|}\Bigr],
\end{align*}
which is exactly $(i)$. 

\medskip
{\sl Point $(ii)$.}
For the second point, we integrate the inequality~\eqref{equal:W} and get:
\begin{align}
Y_t & = Y_0 + \int_0^t  W_s \,ds 
 = Y_0 + (1- e^{-t}) W_0  + \int_0^t (1- e^{s-t}) F(s,Y_s)\,ds  +  N_t, \label{equal:Y} \\
 |Y_t| & \le |Y_0| + |W_0| + \frac t2 + |N_t|, \label{bound:Y}\\
 & \quad \text{with} \quad
N_t := \int_0^t e^{-s} M_s \,ds =  \sqrt 2 \int_0^t (1- e^{u-t})\,dB_u,  \nonumber
\end{align}
where we have used a stochastic version of Fubini's Theorem in the last line.
$N_t$ is a centered random Gaussian variable with variance $\sigma_t^2 := 4e^{-t} - e^{-2t} -3 + 2t \le 2t$. The bound~\eqref{expmomG} and the independence of the initial condition and $N_t$ then lead to
\begin{align*}
\E \Bigl[ e^{\lambda |Y_t|} \Bigr] &  
\le 
  e^{\lambda \frac t 2}    \E \Bigl[ e^{\lambda (|Y_0| + |W_0|)} \Bigr]
\E \Bigl[ e^{\lambda |N_t|} \Bigr], \\
& \le 
  e^{\lambda \frac t 2}   \E \Bigl[ e^{\lambda (|Y_0| + |W_0|)} \Bigr]
\, 2 \, e^{  \frac12 \lambda^2 \sigma_t^2} 
 \le 
2 \, e^{\lambda \frac t 2 + \lambda^2 t }   \, \E \Bigl[ e^{\lambda (|Y_0| + |W_0|)} \Bigr],
\end{align*}
which leads to the claimed result.

\medskip
{\sl Point $(iii)$.} We first integrate equality~\eqref{equal:W} between $s$ and $u$, and  take a supremum in time:
\begin{align}
Y_u &  = Y_s + (1- e^{s-u}) W_s  + \int_s^u (1- e^{v-u}) F(v,Y_v)\,dv  +  N'_u,   \nonumber \\
\sup_{s \le u \le t} |Y_u- Y_s| & \le (t-s) |W_s| + \frac {t-s}2 + \sup_{s \le u \le t} |N'_u|, \quad \text{with} \quad
N_u' :=  \sqrt 2 \int_s^u (1- e^{v-u})\,dB_v,  \label{bound:dY}
\end{align}
But, thanks to the properties of Brownian motion, by a change of time $\tau = \sigma_{t-s}^2 = 4e^{s-t} - e^{2(s-t)} -3 + 2(t-s) \le  (t-s)^3$:
\[
\sup_{s \le u \le t} |N_u'| \stackrel{\LL}= \sup_{0 \le u \le \tau} |B_u|
\stackrel{\LL}= \sqrt \tau \sup_{0 \le u \le 1} |B_u|
.
\]
The law of supremum in time of the absolute value of a 1D Brownian motion is explicitly known, see for instance~\cite[p. 342]{Feller}. Here, we will use only simple estimates on the exponential moments:
\begin{align*}
\E \Bigl[ e^{\lambda \sup_{0 \le u \le 1} |B_u|} \Bigr] & \le  
\E \Bigl[ e^{\lambda \sup_{0 \le u \le 1} B_u} \Bigr] +  \E \Bigl[ e^{\lambda \sup_{0 \le u \le 1} (-B_u)} \Bigr] \\
& \le 2 \, \E \Bigl[ e^{\lambda|B_1|} \Bigr] \le 4 e^{\frac12 \lambda^2}.
\end{align*}
In the second line, we use the well-known equality
$ \sup_{0 \le u \le \tau} B_u \stackrel \LL = \sup_{0 \le u \le \tau} (-B_u)
 \stackrel \LL =  |B_\tau|$,
and then the exponential moments given by~\eqref{expmomG}. The constant $4$ appearing above will raise some difficulties so we will perform a little optimization to get rid of it. For any $\theta \ge 1$, we may also bound
\begin{align*}
\E \Bigl[ e^{\lambda \sup_{0 \le u \le 1} |B_u|} \Bigr]  \le  
\E \Bigl[ e^{\lambda \theta \sup_{0 \le u \le 1} |B_u|} \Bigr]^{\frac 1 \theta} 
\le  \Bigl(4 e^{\frac12 (\theta \lambda)^2} \Bigr)^{\frac1 \theta}
= e^{\frac1 \theta \ln 4 + \frac \theta 2 \lambda^2}
\end{align*}
The optimal $\theta$ seems to be $\theta = \frac{2 \sqrt{\ln 2}}\lambda$. It is admissible when $\lambda \le 2 \sqrt{\ln 2}$. It leads to 
\[
\E \Bigl[ e^{\lambda \sup_{0 \le u \le 1} |B_u|} \Bigr] \le e^{2 \lambda \sqrt{\ln 2} } \le e^{2 \lambda}.
\]
Finally for $t-s \le \lambda^{-2}$, the upper bound on $\tau$ leads to $\lambda (t-s)^{-1} \sqrt \tau \le  \lambda \sqrt{t-s} \le 1$ and 
\[
\E \Bigl[ e^{\lambda (t-s)^{-1}\sup_{s \le u \le t} |N_u'|} \Bigr] =
\E \Bigl[ e^{\lambda (t-s)^{-1} \sqrt \tau \sup_{0 \le u \le 1} |B_u|} \Bigr]
\le e^{2 \lambda \sqrt{t-s}}.
\]
Using  the point (ii) of the Proposition and~\eqref{bound:dY}, we write
\begin{align*}
\E \Bigl[ e^{\lambda (t-s)^{-1} \sup_{s \le u \le t} |Y_u - Y_s|}\Bigr] & \le e^{\frac \lambda 2}\E \Bigl[ e^{\lambda |W_s|} \Bigr]
\E \Bigl[ e^{\lambda (t-s)^{-1}\sup_{s \le u \le t} |N_u'|} \Bigr]   \\
& 
\le e^{\frac \lambda 2} \, e^{\frac \lambda 2(3 + \lambda)} \E \Bigl[ e^{\lambda |W_0|} \Bigr]
\, e^{2 \lambda \sqrt{t-s}},
\end{align*}
which concludes the proof, using that $\sqrt{t-s} \le \frac14$ by assumption.

\medskip
{\sl Point $(iv)$.}
Taking the expectation in~\eqref{bound:W},
\[
\E \bigl[ |W_t| \bigr] \le  e^{-t} \E \bigl[ |W_0| \bigr]  + \frac12  + \E \bigl[|M_t| \bigr].
\]
and using that $ M_t$ is $\NN(0,1-e^{-2t})$ distributed, the expectation is simply bounded: $\E \bigl[|M_t| \bigr] \le \sqrt{2 /\pi} \le 1$ and it implies the first bound of point $(iv)$. The second bound uses~\eqref{bound:Y} written with $s$ instead of $0$, which leads in average to
\[
\E \bigl[ |Y_t - Y_s| \bigr] \le (1- e^{s-t}) \E \bigl[ |W_s| \bigr]  + \frac{t-s}2  + \E \bigl[|N_{t-s}| \bigr].
\] 
Using, $1-e^{s-t} \le t-s$ and the previous bound on $\E[|W_s|]$, and the fact that $N_t$ is $\NN(0,\sigma^2_{t-s})$ distibuted, with $\sigma^2_{t-s} \le  (t-s)^3$, it comes
\[
\E \bigl[ |Y_t - Y_0| \bigr] \le (t-s) \biggl( \E \bigl[ |W_0| \bigr] + 2 + \frac12 + \sqrt{\frac { 2(t-s)} \pi} \biggr),
\]
and the conclusion follows when $t-s \le \frac14$. 
\end{proof}

\subsection{Strong existence via regularization and Feymann-Kac type estimates.} \label{sec:FK}
We introduce a smoothing kernel $\chi \in C^\infty(\R,\R^+)$ with support included in $[-1,1]$ and satisfying $\int_\R \chi(y) \,dy =1$. And then standardly for $\eta >0$,  $\chi_\eta := \eta^{-1} \chi(\frac \cdot \eta)$, and the approximated kernel 
\begin{equation} \label{ker_approx}
K_\eta := K * \chi_\eta, \quad \text{which satisfies } \; 
\left | \left ( K_{\eta}-K \right )(x) \right | \leq \indiq_{[-\eta,\eta]}(x),
\end{equation}
for all $x \in \R$. Given a stochastic basis, and a Brownian $(B_t)_{t \ge 0}$ motion on it, we consider the following non linear SDE:
\begin{equation} \label{SDE_approx}
Y^\eta_t= Y_0^\eta + \int_0^t W^\eta_s \,ds, \qquad 
W^\eta_t= W_0^\eta+  \int_0^t \E_{\overline{Y}}\bigl[K_\eta(Y^\eta_s-\overline{Y}) \bigr]\, ds  - \int_0^t W^\eta_s \, ds +   \sqrt 2\, B_t ,
\end{equation}
where $\overline{Y}$ is an independent copy of $Y^\eta$ and the initial condition $(Y_0^\eta,W_0^\eta)$  is defined as
\begin{equation}
\label{CI_eta}
Y_0^{\eta}:=Y_0+\eta U, \qquad W_0^{\eta}:=W_0+\eta V,
\end{equation}
where $(Y_0,W_0)$ has law $f_0$ and is independent of $(U,V)$ of law $\chi\otimes\chi$ (and both are independent of the Brownian motion $(B_t)_{t\geq 0}$). Then $(Y_0^\eta,W_0^\eta)$ has for law $\mu_0^\eta:= f_0 * \tilde \chi_\eta$ with $\tilde \chi_\eta(y,w) := \chi_\eta(y) \chi_\eta(w)$, and is independent of the Brownian motion $(B_t)_{t\geq 0}$ . Introducing the notation
\begin{equation*}
\widetilde{K}_{\eta}[\mu](x)=\int_{\R^2} K_\eta(x-y)\mu(dy,dw),
\end{equation*}
this system can be written in an equivalent manner:
\begin{equation*}
Y^\eta_t= Y_0^\eta + \int_0^t W^\eta_s \,ds, \qquad 
W^\eta_t= W_0^\eta +  \int_0^t \widetilde{K}_{\eta}[\mu^{\eta}_s](Y^\eta_s)\, ds  - \int_0^t W^\eta_s \, ds + \sqrt 2\,  B_t ,
\end{equation*}
where $\mu^\eta_t$ is the time marginal at time $t$, \emph{i.e.\ }the law of $(Y^\eta_t,W^\eta_t)$.

\medskip
Since the kernel $K_\eta$ is globally Lipschitz, \cite[Thm 1.1]{sznitman}  implies the strong existence and uniqueness of the process $(Y^\eta_t,W^\eta_t)_{t \ge 0}$ solving~\eqref{SDE_approx}. And by an application of Ito's rule the family of the time marginals $(\mu^\eta_t)_{t \ge 0}$ of that process is a weak solution of the following regularized Vlasov-Poisson-Fokker-Planck equation:
\begin{equation}
\frac{\partial}{\partial t}\mu^\eta_t+v \, \partial_x \mu^\eta_t+\widetilde{K}_\eta[\mu^\eta_t] \, \partial_v\mu^\eta_t= \partial_v ( \partial_v \mu^\eta_t + v \mu^\eta_t ),
\label{eq:reg VPFP}
\end{equation}  
with the initial condition $\mu^\eta_0=\LL(Y_0,W_0)=f_0*\tilde \chi_{\eta}$. We begin by proving some $\eta$ independent estimates on $\mu^\eta_t$, for $t\ge 0$.

\paragraph{Feynmann-Kac type estimates}

\begin{lem}\label{lem:FK}
Assume that the law of the initial condition of equation \eqref{eq:NLSDE} $f_0 \in \PPP_1 \cap L^1(\R^2)$ satisfies either
$\| f_0\|_{e,\lambda} < \infty$ for some $\lambda >0$ or 
$\| f_0\|_{p,\gamma} < \infty$ for some $\gamma >1$.
 Then for all $t>0$, the unique (measure) solution to the smoothed VPFP equation~\eqref{eq:reg VPFP} with initial condition $\mu^\eta_0$ satisfies respectively
 \begin{equation} \label{def:Clambda}
 \| \mu^{\eta}_t\|_{e,\lambda} \le  2   \, e^{ t+ \lambda \eta + \frac{\lambda}2 +  \frac{\lambda^2}2} \, \| f_0 \|_{e,\lambda e^{-t}} ,
 \qquad
 \| \mu^{\eta}_t\|_{p,\gamma} \le \| f_0\|_{p,\gamma}
\left\langle \eta \right\rangle^{\gamma} C_{\gamma} e^{(1+\gamma) t},
 \end{equation}
 where $C_{\gamma}$ is a constant depending explicitly on $\gamma$.
 
In particular, the associated spatial density $\rho^\eta_t := \int_\R \mu^\eta_t(x,v) \,dv$ satisfies respectively
 \begin{equation} \label{def:Klambda}
\left\|\rho^\eta_t \right\|_{\infty}  \leq \frac4 \lambda \, e^{ t+ \frac{\lambda}2 +  \frac{\lambda^2}2} e^{\lambda \eta} \,   \| f_0 \|_{e,\lambda e^{-t}} ,
\qquad 
\| \rho^\eta_t \|_{\infty}  \leq   \frac {2\gamma}{\gamma-1}  C_\gamma
e^{(1+ \gamma) t}  \la \eta\ra ^\gamma \| f_0\|_{p,\gamma}.
\end{equation}

\end{lem}

\begin{proof}
{\sl Step 1. Regularization and Feynmann-Kac's formula.}
Fix $t \ge 0$ and consider the following ``backward'' SDE:
\begin{equation} \label{back_SDE}
Y^{x,v}_s  = x - \int_0^s W^{x,v}_u \,du, \qquad
W^{x,v}_s  = v - \int_0^s \widetilde K_\eta[\mu^\eta_{t-u}](Y^{x,v}_{u})\, du  + \int_0^s W^{x,v}_u \, du + \sqrt 2\,  B_s ,
\end{equation}
First note that  $\tilde K_\eta [\mu^\eta]$ is uniformly Lipschitz in position on $\R^+ \times \R^2$.  So
strong existence and uniqueness of solution to the (linear) SDE~\eqref{back_SDE} are guaranteed by standard results.  We set:
\begin{equation*}
\theta_s= e^ s\, \mu^\eta_{t-s} \bigl(Y_s^{x,v},W_s^{x,v} \bigr).
\end{equation*}

Moreover the initial condition $\mu^\eta_0 = f_0 * \tilde \chi_\eta$ fulfills the hypothesis of Proposition~\ref{prop:app} of the Appendix: $\partial_x^k \partial_v^l \mu^{\eta}_0 \in L^2(\R^2)$ for any $k,l \ge 0$. This implies that $\mu^{\eta}_t(x,v)$ possesses one continuous derivative in time, and two (continuous) derivative in position and velocity. So, we may apply Ito's rule to $\theta$: we get
\begin{equation*}
\begin{split}
e^{-s} d\theta_s & = \mu^\eta(t-s,Y_s^{x,v},W_s^{x,v})ds - \partial_t\mu^\eta(t-s,Y_s^{x,v},W_s^{x,v})ds+\partial_x\mu^\eta(t-s,Y_s^{x,v},W_s^{x,v})dY_s^{x,v}\\
& \hspace{30mm} +\partial_v \mu^\eta (t-s,Y_s^{x,v},W_s^{x,v})   dW_s^\eta+ \Delta_v \mu^\eta (t-s,Y_s^{x,v},W_s^{x,v})\left\langle dW_s^{x,v}\right\rangle^2\\
&=\left [ - \partial_t \mu^\eta -v \, \partial_x\mu^\eta - \widetilde K_\eta[\mu^\eta] \, \partial_v \mu^\eta+ \partial_v (v \mu^\eta) +  \Delta_v \mu^\eta  \right ](t-s,Y_s^{x,v},W_s^{x,v})ds\\
& \hspace{30mm}+\partial_v\mu^\eta(t-s,Y_s^{x,v},W_s^{x,v})dB_s,
\end{split}
\end{equation*}
and since  $\mu^\eta$ is a strong solution of~\eqref{eq:reg  VPFP}, we get precisely that for any $0 \le s \le s' \le t$:
\begin{equation*}
\theta_{s'} - \theta_s= \int_s^{s'} e^u \partial_v\mu^\eta(t-u,Y_u^{x,v},W_u^{x,v}) \, dB_u.
\end{equation*}
In particular, $(\theta_s)_{0 \le s  \le t }$ is a martingale, so that
\begin{equation} \label{FKform}
\mu^\eta(t,x,v)=  \theta_0 = \E [ \theta_t ] = e^t \E \bigl[ f_0* \tilde \chi_{\eta}(Y_t^{x,v},W_t^{x,v})  \bigr].
\end{equation}

\medskip
{\sl Step 2. Proof in the case of uniform exponential tails.}

But by the hypothesis on $f_0$, and since $\chi$ has support in $[-1,1]$,
\begin{equation}
\begin{split}
\label{eq:MolExp}
f_0* \tilde \chi_{\eta}(x,v)&= \int f_0(x-y,v-w) \tilde \chi_{\eta}(y,w) \, dydw
 \leq    \| f_0\|_{e,\lambda} \int_\R  e^{-\lambda | v-\eta w' | }\chi(w') \, dw' \\
&  \leq \| f_0\|_{e,\lambda} e^{-\lambda |  v | } \int_\R e ^{\lambda \eta |w'|}\chi(w') \, dw' \leq  \| f_0\|_{e,\lambda} e^{-\lambda |v|+ \lambda \eta}.
\end{split}
\end{equation}
Moreover, the definition~\eqref{back_SDE} of $W^\eta$ also implies that for $0 \le s \le t$:
\[
W_s^{x,v}= e^s v - \int_0^s e^{s-u}  \widetilde{K_\eta}[\mu^\eta_u] \, du + M_s, 
\quad  \text{with }\; M_s :=  \sqrt 2 \int_0^s e^{s- u} dB_u,
\]
Remark that $M_s$ is in fact a centered Gaussian variable with variance $e^{2s}-1$. Since $ \bigl\| \widetilde{K^\eta}[\mu^\eta_u]  \bigr\|_\infty\le 1/2$, it leads to the following lower bound 
\begin{equation} \label{lowbd:W} 
| W^{x,v}_t |  \ge e^t| v| - |M_t|  - \frac{e^t}{2}.
\end{equation}
Using all of this in the representation formula~\eqref{FKform} leads to
\[
 \mu^\eta_t (x,v) \le \| f_0\|_{e,\lambda} e^{t+ \lambda \eta} \, \E \Bigl[  e^{ - \lambda |W_t^{x,v}|} \Bigr]
\le \| f_0\|_{e,\lambda} e^{ t+ \lambda (e^t/2 + \eta)  - \lambda  e^t  |  v |} \, \E\left[  e^{ \lambda |M_t|} \right]
 \]
An application of~\eqref{expmomG} to $M_t$ leads to the following bound that is uniform in $\eta$ (for $\eta$ small) : 
\begin{equation} \label{up_bound}
\| \mu^\eta_t \|_{e,\lambda e^t} \leq   2 \, e^{t+\lambda \eta + \frac{\lambda}2e^t +  \frac{\lambda^2}2 }  \| f_0\|_{e,\lambda}.
\end{equation}
The conclusion follow by replacing $\lambda$ by $\lambda e^{-t}$ in the above bound. And the estimate on $\| \rho^\eta_t \|_\infty$ is simply obtained by integration on $v$.

\medskip
{\sl Step 3. Proof in the case of uniform polynomial tails.} 

The simple inequality $\la v+w \ra \leq \sqrt 2 \la v \ra  \, \la w \ra$  (recall that $\la v \ra^2 := 1 + v^2$) implies that
$  \la v-w \ra^{-1} \leq \sqrt 2 \la v \ra^{-1}  \, \la w \ra$, which allows to bound
\begin{align*}
f_0* \tilde \chi_{\eta}(x,v)&= \int f_0(x-y,v-w) \tilde \chi_{\eta}(y,w) \, dydw
 \le   \| f_0\|_{p,\gamma}  \int_\R \la  v-\eta w' \ra^{- \gamma}  \chi(w') \, dw' , \\
 & \le  2^{\gamma/2}  \| f_0\|_{p,\gamma}  \la v  \ra^{- \gamma} \int_\R \la \eta w' \ra^\gamma  \chi(w') \, dw' 
 \le 2^{\gamma/2}  \| f_0\|_{p,\gamma}  \la \eta \ra^\gamma  \la v  \ra^{- \gamma}.
\end{align*}
Plugging it into the representation formula~\eqref{FKform}, using the lower bound~\eqref{lowbd:W} and the simple inequality,  
$\la v + \frac 12 e^t  \ra^\gamma \le 2^{\gamma/2} \bigl(e^{\gamma t}  + (2v)^\gamma \bigr)$ (simply separate the cases $v \le \frac 12 e^t$ and $v \ge \frac 12 e^t$),
\begin{align*}
\mu^\eta(t,x,v) & =e^t \, \E \bigl[ f_0* \tilde \chi_{\eta}(Y_t^{x,v},W_t^{x,v})  \bigr]
\leq  2^{\gamma/2}  \| f_0\|_{p,\gamma}  \la \eta \ra^\gamma  \E \bigl[ \la W_t^{x,v}  \ra^{- \gamma} \bigr] \\
& \leq 2^{\gamma/2}  \| f_0\|_{p,\gamma}  \la \eta \ra^\gamma  \E \biggl[ \Bigl\langle e^t| v| - |M_t|  - \frac{e^t}2  \Bigr \rangle^{- \gamma} \biggr] \\ 
&\leq 2^\gamma  \| f_0\|_{p,\gamma} \la \eta \ra^\gamma \bigl\langle  e^t v  \bigr\rangle^{-\gamma} \, 
\E \Bigl[ \Bigl\langle |M_t| + \frac{e^t}2  \Bigr\rangle^\gamma \Bigr]
\\
&\leq 2^{3\gamma/2}  \| f_0\|_{p,\gamma} \la \eta \ra^\gamma \bigl\langle v  \bigr\rangle^{-\gamma}
\, \Bigl(   e^{\gamma t}  + 2^\gamma \E \bigl[ |M_t|^\gamma \bigr] \Bigr).
\end{align*}
Since $M_t\sim \mathcal{N}(0,e^{2t}-1)$ we have
$ \E\bigl[ | M_t  |^{\gamma} \bigr]=(e^{2t}-1)^{\frac{\gamma} {2}} \, m_{\gamma} \le e^{\gamma t}  m_\gamma$, 
where $m_{\gamma}$ stands for the moment of order $\gamma$ of the law $\mathcal{N}(0,1)$. 
This implies the claimed bound on $\|\mu^{\eta}_t\|_{p,\gamma}$ with $C_\gamma := 2^{3\gamma/2} (1+ 2^\gamma m_\gamma)$. The bound on $\| \rho^{\eta}_t\|_\infty$ follows from the simple bound $ \la v \ra^{- \gamma} \le \min\bigl(1,|v|^{-\gamma}\bigr)$ and an integration  in $v$.
\end{proof}

\paragraph{Completeness estimates.}
Thanks to the propagation of the uniform estimates on the tails in velocity and the crucial Lemma~\ref{lem:rope}, we are in position to show that the family $( Y^\eta_t, W^\eta)_{\eta > 0}$  of solutions to~\eqref{SDE_approx} has the Cauchy property.

\begin{lem} \label{lem:CaucEst}
For $\eta$,$\eta'>0$, let $( Y^\eta_t, W^\eta_t)_{t\ge 0}$ and $( Y^{\eta'}_t, W^{\eta'}_t)_{t \ge 0}$ be two (unique) solutions of the nonlinear SDE~\eqref{SDE_approx}, constructed on a given probability basis, with a common Brownian motion $(B_t)_{t \ge 0}$,  and initial condition chosen as~\eqref{CI_eta}. If the law of the initial condition satisfies either $\| f_0\|_{e,\lambda} < \infty$ for some $\lambda >0$ or $\| f_0\|_{e,\gamma} < \infty$ for some $\gamma >1$, then the following stability estimate holds for any $t \ge 0$:
\begin{equation} \label{stab_approx}
\E\left[\sup_{s \in [0,t]}\left (  \left | Y^\eta_s-Y^{\eta'}_s \right |+\left | W^\eta_s-W^{\eta'}_s \right | \right ) \right]
\leq \frac32(\eta+ \eta')
\exp \Bigl( (1+ 8 K_{t, \eta+\eta'}) t \Bigr),
\end{equation}
where $K_{t, \eta+\eta'}$ is the constant appearing in the r.h.s of~\eqref{def:Klambda} with $\eta$ replaced by $\eta + \eta'$, \emph{i.e.\ }respectively 
\[
K_{t, \eta+\eta'} =
 \frac4 \lambda \, e^{ t+ \frac{\lambda}2 +  \frac{\lambda^2}2} e^{\lambda (\eta+ \eta')} \,   \| f_0 \|_{e,\lambda e^{-t}} ,
\quad  \text{or} \quad 
K_{t, \eta+\eta'} =   \frac {2\gamma}{\gamma-1}  C_\gamma
e^{\gamma t}  \la \eta + \eta'\ra ^\gamma \| f_0\|_{p,\gamma}.
\]
\end{lem}

\begin{proof} Our strategy is the same as in the proof of the weak strong stability estimate in subsection~\ref{sec:WSSU}. 
For $t\in [0,T]$ we have:
\begin{align*}
\sup_{s\in [0,t]}\left |  Y^\eta_t-Y^{\eta'}_t \right | & \leq \int_0^t \left | W^\eta_s-W^{\eta'}_s \right | \, ds \, + |\eta-\eta'||U|\\
\sup_{s\in [0,t]} \left | W^\eta_s-W^{\eta'}_s \right | & \leq \int_0^t \left|\widetilde{K}_{\eta}[\mu^\eta_s](Y^\eta_s)-\widetilde{K}_{\eta'}[\mu^{\eta'}_s](Y^{\eta'}_s) \right | \,ds \, + |\eta-\eta'||V| \\
& \leq \int_{0}^t  \biggl(\left | \widetilde{K}_{\eta}[\mu^\eta_s](Y^\eta_s)-\widetilde{K}[\mu^\eta_s](Y^\eta_s)\right |
+\left |\widetilde{K}[\mu^\eta_s](Y^\eta_s)-\widetilde{K}[\mu^{\eta'}_s](Y^{\eta'}_s)\right | \\
& \hspace{20mm} +\left | \widetilde{K}[\mu^{\eta'}_s](Y^{\eta'}_s) -\widetilde{K}_{\eta'}[\mu^{\eta'}_s](Y^{\eta'}_s) \right | \biggr) \, ds  + |\eta-\eta'|,
\end{align*}
since $|V|$ and $|U|$ are always bounded by $1$ (we recall that $\chi$ has its support included in $[-1,1]$). 
But thanks to~\eqref{ker_approx}, for any $y \in \R$:
\begin{equation*}
 \left | \widetilde{K}_{\eta}[\mu^\eta_s](y)-\widetilde{K}[\mu^\eta_s](y) \right | = \left |\int (K_\eta-K)(y-y') \, \rho^\eta_s(dy') \right |  \leq \int \indiq_{[-\eta,\eta]}(y-y') \,\rho^\eta_s(dy') 
\leq 2\eta\left \| \rho^\eta_s \right \|_{\infty}. 
\end{equation*}
This allows to bound the the first and third term, in the r.h.s.\  of the second inequality by $2 \int_0^t \bigl( \eta  \| \rho^\eta_s \|_{\infty}+ \eta'  \| \rho^{\eta'}_s  \|_\infty \bigr) \,ds$.  And the second term is estimated in expectation with the help of Lemma~\ref{lem:rope}:
\begin{align*}
\E \Bigl[ \left |\widetilde{K}[\mu^\eta_s](Y^\eta_s)-\widetilde{K}[\mu^{\eta'}_s](Y^{\eta'}_s)\right | \Bigr]
& \le 8  \min \bigl( \|\rho^\eta_t \|_\infty , \|\rho^{\eta'}_t \|_\infty \bigr) \, \E \bigl[   |Y^\eta_t - Y^{\eta'}_t| \bigr] \\
& \le 4 \bigl( \|\rho^\eta_t \|_\infty + \|\rho^{\eta'}_t \|_\infty \bigr) \, \E \bigl[  | Y^\eta_t - Y^{\eta'}_t| \bigr].
\end{align*}
Gathering all of this leads to
\begin{equation*} 
 \E\left [\sup_{s\in[0,t]}\left (  \bigl| Y^\eta_s-Y^{\eta'}_s \bigr|+\bigl| W^\eta_s-W^{\eta'}_s \bigr | \right )\right ] 
 \leq \int_0^t  \alpha(s) \biggl(  \E\Bigl [ \bigl| Y^\eta_s-Y^{\eta'}_s \bigr|+\bigl| W^\eta_s-W^{\eta'}_s \bigr | \Bigl] + \frac32(\eta+ \eta') \biggr)\, ds,
\end{equation*}
with $\alpha(s) := 1+4  \| \rho_s^\eta\|_\infty + 4  \| \rho_s^{\eta'}\|_\infty$.  A simple application of  the Gronwall's lemma to $\E \bigl [\sup_{s\in[0,t]}\bigl ( | Y^\eta_s-Y^{\eta'}_s|+| W^\eta_s-W^{\eta'}_s  | \bigr )\bigr ]  + 3(\eta+\eta')/2$ leads to
\[\E\biggl[ \;\sup_{s\in [0,t]}\left (   \bigl| Y^\eta_s-Y^{\eta'}_s \bigr|+\bigl| W^\eta_s-W^{\eta'}_s \bigr | \right )\biggr ] 
 \le \frac32(\eta+ \eta')     \exp \biggl(\int_0^t \alpha(s)\,ds\biggr).
\]
The conclusion follows from~\eqref{def:Klambda}, which implies that
$\| \rho^\eta_s\|_\infty \le  K_{s,\eta} \le K_{t,\eta+\eta'}$
for any $s \in[0,t]$ (and a similar inequality with $\eta'$ replacing $\eta$).
\end{proof}

\paragraph{The Cauchy property in the space of path.}

We now consider the space $\mathcal A$ of measurable applications (or random variables) from $\Omega=C(\R^+;\R^2)$ (with the Wiener measure) into itself. We endow it with the topology of uniform convergence on compact (in time) subsets. An associated distance to that topology is for instance 
\begin{equation} \label{def:d}
\forall\, Y,Z \in \mathcal A, \quad 
d(Y,Z)=  \sum_{n \in \N^*} \frac1{2^n}  \E\left( 1 \wedge \sup_{t \in [0,2^n]} \left( |Y^1_t-Z^1_t|+|Y^2_t-Z^2_t|\right)\right),
\end{equation}
for which $\mathcal A$ is complete (the symbol $\wedge$ stand for the minimum). Let us consider the ``sequence'' $(Y^\eta,W^\eta)_{\eta >0}$ (the correct denomination is ``net'').  By Lemma~\ref{lem:CaucEst} it is a Cauchy ``sequence'' ( or net) in $(\mathcal A,d)$, and then it converges towards a certain $(Y_t,W_t)_{t \ge 0}$ in $(\mathcal A,d)$. 

At a fixed time $t$, this implies the convergence in probability and then in law of $(Y^\eta_t,W^\eta_t)$ toward 
$(Y_t,W_t)$: \emph{i.e.\ }the time marginals $\mu^\eta_t$ weakly converge (as measures) towards $f_t$.  Using a standard argument, we can pass in the limit in the uniform bound~\eqref{def:Clambda}
obtained in Lemma~\ref{lem:FK}. So that for any time $t$, the density of the law $f_t$ of $(Y_t,W_t)$ satisfies 
one of the bound of~\eqref{eq:expdec}.

\paragraph{Identification of the limit.}

In order to prove that $(Y_t,W_t)_{t \ge 0}$ is a solution  to~\eqref{eq:NLSDE} we have to show that for any $t \ge 0$:
\begin{equation} \label{recog}
\E \biggl[  \sup_{s  \in [0,t]}   \Bigl |  Y_s - Y_0 - \int_0^s W_u \,du \Bigr| \biggr]   =0, \quad
\E \biggl[  \sup_{s  \in [0,t]}   \Bigl |  W_s - W_0 + \int_0^s \bigl(W_u - \widetilde K[\mu_u](Y_u) \bigr) \,du - B_s \Bigr|  \biggr] =0.
\end{equation}
But this is something we know for the approximated process $(Y_t^\eta,W_t^\eta)_{t \ge 0}$. Precisely, by Definition of strong solution to~\eqref{SDE_approx}:
\begin{multline*}
\E \biggl[  \sup_{s  \in [0,t]}   \Bigl |  Y^\eta_s - Y_0  -\eta U - \int_0^s W^\eta_u \,du   \Bigr| \biggr]    =0, \\
\E \biggl[  \sup_{s  \in [0,t]}   \Bigl |  W^\eta_s - W_0 -\eta V + \int_0^s \bigl(W^\eta_u - \widetilde K_\eta[\mu^\eta_u](Y^\eta_u) \bigr) \,du - B_s \, \Bigr|  \biggr] =0.
\end{multline*}
But the convergence in $(\mathcal A,d)$ allows to pass to the limit in the first equality above , and we obtain the first equality in~\eqref{recog}. In order to pass to the limit in the second inequality, and get the second part of~\eqref{recog}, the main difficulty is to handle the non-linear term. Precisely, we will show in the rest of the proof that
\begin{equation*}
\sup_{s \in [0,t]}  \E\left ( \left | \widetilde{K}_\eta[\mu^\eta_s](Y^\eta_s)- \widetilde{K}[\mu_s](Y_s) \right |\right )\xrightarrow[\eta \rightarrow 0]{} 0.
\end{equation*}
This will conclude the proof.

For this, we introduce a independent copy $(\bar Y^\eta_t, \bar Y_t)$ of the couple $( Y^\eta_t, Y_t)$. We  then rewrite the force with the help of that independent copy and estimate
\begin{align*}
\E \Bigl[ \left | \widetilde{K}_\eta[\mu^\eta_s](Y^\eta_s)- \widetilde{K}[\mu_s](Y_s) \right | \Bigr] & = 
\E \Bigl[ \bigl| K_\eta(Y^\eta_s - \bar Y^\eta_s) - K(Y_s - \bar Y_s) \bigr| \Bigr] \\
& \le \E \Bigl[ |K_\eta- K| (Y^\eta_s - \bar Y^\eta_s)  \Bigr] + 
\E \Bigl[ \bigl| K(Y^\eta_s - \bar Y^\eta_s) - K(Y_s - \bar Y_s) \bigr| \Bigr]
\end{align*}
The first term in the r.h.s is bounded thanks to~\eqref{ker_approx}:
\begin{align*}
\E \bigl[ |K_\eta- K| (Y^\eta_s - \bar Y^\eta_s)  \bigr]  & \le \E \bigl[  \indiq_{[-\eta,\eta]}(Y^\eta_s - \bar Y^\eta_s)  \bigr]
\le \E \Bigl[   \E \bigl[ \indiq_{[-\eta,\eta]}(Y^\eta_s - \bar Y^\eta_s) \big|  Y^\eta_s \bigr]  \Bigr] \\
& \le   \E \bigl[   2 \eta \|  \rho^\eta_s \|_\infty  \bigr]  = 2 \eta \|  \rho^\eta_s \|_\infty
\end{align*}
The second term in the r.h.s. is bounded thanks to Lemma~\ref{lem:rope} and we get
\[
\E \Bigl[ \left | \widetilde{K}_\eta[\mu^\eta_s](Y^\eta_s)- \widetilde{K}[\mu_s](Y_s) \right | \Bigr] 
\le  2 \eta \|  \rho^\eta_s \|_\infty + 8\,  \| \rho_s\|_\infty  \E \bigl[  |Y^\eta_s - Y_s | \bigr].
\]
Since $\| \rho^\eta_s\|_\infty$ and $\| \rho_s\|_\infty$ are bounded uniformly in time on $[0,t]$ and for $\eta \in (0,1)$, 
 and thanks to the convergence of $(Y^\eta,W^\eta)$ towards $(Y,W)$ in $(\mathcal A,d)$, it is simple to conclude that the requested term goes to zero, as $\eta$ goes to zero.

%
%
\subsection{Proof of Corollary~\ref{prop:WPlaw}} \label{sec:WPlaw}

The existence part is a simple consequence of the existence of the process solution to the non linear SDE~\eqref{eq:NLSDE}, for a given initial condition with a polynomial or exponential decays of the velocity tails (use Ito's rule). So we only prove the weak-strong stability estimate.

Let $(g_t)_{t \ge 0}$ be weak solution to~\eqref{eq:VPFP}  starting from $g_0\in\mathcal{P}_1(\mathbb{R}^2)\cap L^1(\mathbb{R}^2)$ and $(f_t)_{t \ge 0}$ be another weak solution to ~\eqref{eq:VPFP} starting from $f_0\in \mathcal{P}_1(\mathbb{R}^2)\cap L^1(\mathbb{R}^2)$ satisfying for any $t>0$
\[
\sup_{s\leq t} \left\|\rho_s\right\|_{\infty}<\infty,
\qquad \text{where} \quad 
\rho_s(x) =\rho_s^f := \int_\R f_s(x,dv).
\]

By~\cite[Theorem 2.6]{Figalli} or~\cite[Proposition B.1]{FouHau}, there exists a stochastic basis $\left ( \Omega,\mathbb{P},\left (\mathcal{F}_t \right )_{t\geq 0},\mathcal{F} \right )$ and a Brownian motion $(B_t)_{t \ge 0}$ on this basis and a process $(X_t,V_t)_{t \ge 0}$ solution to~\eqref{eq:NLSDE}, which has exactly the time marginal $g_t$ at any time $t \ge 0$. \\
 
Next, remark that the force field $\widetilde K$ created by $f$, precisely $\widetilde K_s(x) = \int_\R K(x-y) \rho^f_s(dy)$,  is Lipschitz in position. In fact,
\[
 \bigl| \widetilde K_s(x)-\widetilde K_s(y) \bigr|  \le   \int_\R  \bigl| K(x-z)-K(y-z)  \bigr| \rho_s(dz) 
  \le   \int_\R  \indiq_{[x,y]} (z)  \rho_s(dz) \le 
  \|\rho_s \|_\infty |x-y|.
  \]
 Extending the probability space, we may choose  a r.v. $(Y_0,W_0)\sim f_0$ such that
 \[
 W_1(f_0,g_0)=\E\Bigl[ \bigl| Y_0-X_0 \bigr| +\bigl| W_0-V_0 \bigr| \Bigr],
 \]
 and since $\tilde K$ is regular, standard results allow us to build on that probability space a 
 stochastic process $(Y_t,W_t)_{t \ge 0}$ solution to
 \[
 Y_t=Y_0+\int_0^t W_s \, ds,  \qquad 
 W_t=W_0+\int_0^t\widetilde K_s(Y_s)\, ds-\int_0^t W_s \,ds+\sqrt 2 \, B_t.
 \]
Note that  the family of time marginals $(h_t)_{t\geq 0}$ of  $(Y_t,W_t)$ is  a solution to the following linear Vlasov-Fokker-Planck equation
 \begin{equation}
 \label{eq:LVFP}
 \partial_t h_t+v \, \partial_x h_t+\widetilde K_t(x) \,\partial_v h_t=\partial_v(\partial_v h_t+vh_t), 
 \quad \text{where} \quad \widetilde{K}_s(x)=K * \rho^f_s(x),
 \end{equation}
 with initial condition $h_{t=0}=f_0$. Of course, that equation is also satisfied by $f$ by assumption. But since  $\tilde K$ is globally Lipschitz in the space variable, uniformly in the time variable,  uniqueness holds for equation~\eqref{eq:LVFP} in the class $L^\infty_t(L^1)$ by standard results. So, we have $h_t=f_t$ for all time $t \ge 0$ and
$(Y_t,W_t)_{t\geq 0}$ is actually a solution to~\eqref{eq:NLSDE} defined on the same probability space as $(X_t,V_t)_{t \ge 0}$. Then applying the point $(ii)$ of Theorem~\ref{thm:WPproc} to those processes and recalling that
 \[
 W_1(f_t,g_t)\leq \E\Bigl[ \bigl|Y_t-X_t \bigr|+\bigl|W_t-V_t \bigr| \Bigr],
 \]
we get the expected estimate.

%
%
%
%
%
%
%
%
%
\section{Proof of Theorem~\ref{thm:MeaPC}}  \label{sec:MeaPC}
Before going to the proof, we will prove a useful lemma.
\begin{lem}
\label{lem:(De)-Poiss}
Let $(X_1,\ldots,X_N)$ be $N$ i.i.d random variables of law  $\rho\in \mathcal{P}(\R^d)$,  and $\rho_N= \frac1N \sum_i \delta_{X_i}$ be the associated empirical measure. Then, for all $a\in \R^d$ we have:
\begin{equation*}
\E\biggl[ \, \sup_{u\in\R_+} \, \biggl| \, \int_{\R^d} \indiq_{|a-y|\leq u} \, (\rho_N-\rho)(dy)\biggr|\biggr] \, \leq \, \frac{3} {\sqrt N}
\end{equation*}
\end{lem}

\begin{proof}
In fact, we choose a sequence $(X_n)_{n\in\N}$  of i.i.d random variables of law $\rho\in \mathcal{P}(\R^d)$ and $L$ some Poisson random variable of parameter $N$ independent of  the $(X_n)_{n\in\mathbb{N}}$. We define the 
two point process by 
\begin{equation*}
M_N=\sum_{i=1}^L\delta_{Y_i},
\end{equation*}
$M_N$ is in fact a Poisson Random Measure (PRM) with  intensity measure $N\rho$.
Remark that $\left\|M_N-N\rho_N\right\|_{TV} = |L-N|$. For all $a\in \R^d$, we have
\begin{align}
\left|\int \indiq_{|a-y|\leq u}(\rho_N-\rho)(dy)\right|&\leq \left|\int \indiq_{|a-y|\leq u}\biggl(  \frac1N M_N- \rho \biggr)(dy)\right| + \left|\int \indiq_{|a-y|\leq u} \biggl(  \frac1N M_N- \rho_N \biggr)(dy)\right| \nonumber\\
&\leq \frac{1}{N} \left|\int \indiq_{|a-y|\leq u} \bigl(M_N - N \rho \bigr) (dy)\right|
+\frac{1}{N}  \bigl\| M_N-N \rho_N\bigr\|_{TV}
\nonumber\\
\sup_{u\in\R_+}  \biggl|\int \indiq_{|a-y|\leq u}(\rho_N-  \rho)(dy)\biggr| &
\leq  \frac{1}{N}\sup_{u\in\R_+}\left|\MM^{N,a}_u\right|+\frac{|L-N| }{N},  \label{bound:mart}\\
& \hspace{30mm} \text{where} \quad 
\MM^{N,a}_u=\int \indiq_{|a-y|\leq u} \bigl(M_N - N \rho \bigr)(dy). \nonumber 
\end{align}
Since $M_N$ is a PRM, $(\MM^{N,a}_u)_{u\geq 0}$ is a martingale with respect to the filtration $(\FF^a_{u})_{u\geq 0}=  \bigl(\sigma\bigl( M_N \indiq_{B(a,u)} \bigr) \bigr)_{u\geq 0}$,
where $B(a,u)$ the ball of center $a$ and radius $u$ in $\R^d$. So using Doob's inequality and the fact that $
\MM^{N,a}_\infty = L-N$ we find
\begin{equation*}
\E\biggl[\sup_{u\geq 0}\bigl| \MM^{N,a}_u\bigr| \biggr] \leq 
\left( \E\biggl[\sup_{u\geq 0}\bigl| \MM^{N,a}_u\bigr|^2 \biggr]\right)^{1/2}\leq 2\left( \E\Bigl[\bigl| \MM^{N,a}_\infty \bigr|^2\Bigr]\right)^{1/2}\leq  
2 \, \V[L] =  2\sqrt {  N}.
\end{equation*}
Taking now the expectation in~\eqref{bound:mart}, using the above bound and  $\E |L-N| \leq \sqrt{ \V(L)}= \sqrt N$, we conclude the proof.
\end{proof}

We are now in position to prove Theorem~\ref{thm:MeaPC}.

\medskip
{\sl Step 1. Coupling estimates.}

We begin by a calculation valid for any fixed realization of these processes (\emph{i.e.\ }given any initial conditions and Brownian paths):
for all $i=1,\cdots,N$ we have:
\begin{align*}
\sup_{s\in [0,t]}\left | \Xis - \Yis \right | & \le 
\left | X^N_{i,0} - Y^N_{i,0}\right | +  \int_0^t \bigl|V_{i,s}^N-W_{i,s}^N \bigr|\, ds \\
\sup_{s\in [0,t]}\left | \Vis - \Wis \right |
&  \le 
\left | V^N_{i,0} - W^N_{i,0}\right | +
\int_0^t\left | \frac{1}{N}\sum_{j=1}^{N}K(X^N_{i,s}-X^N_{j,s})-\left ( \int_{\R \times \R } K(Y^N_{i,s}-x)\mu_t(dx,dv) \right ) \right |  \\ 
& \le 
\left | V^N_{i,0} - W^N_{i,0}\right | +
\int_0^t \biggl| \frac1N \sum_{j=1}^{N}K(X^N_{i,s}-X^N_{j,s})-K(Y^N_{i,s}-Y^N_{j,s})\biggr| \,ds + \frac t{N-1} + \int_0^t \Lambda^N_{i,s} \,ds 
\\  \text{with}  &\quad \Lambda_{i,s}^N :=  \biggl | \frac{1}{N-1}\sum_{j\neq i}^{N}K(Y^N_{i,s}-Y^N_{j,s})- \int_{\R \times \R } K(Y^N_{i,s}-x)\mu_s(dx,dv) \biggr|
\end{align*}
We sum these inequalities over $i=1,\cdots,N$, divide by $N \ge 2$ and then get:
\begin{equation} \label{eq:MKE}
\begin{split}
\frac{1}{N}\sum_{i=1}^N\sup_{s\in[0,t]}\left | \Xis-\Yis \right| & \leq  
\frac{1}{N}\sum_{i=1}^N \left | X^N_{i,0} - Y^N_{i,0}\right | + 
\int_0^t\frac{1}{N}\sum_{i=1}^N|V_{i,s}^N-W_{i,s}^N|ds\\
\frac{1}{N}\sum_{i=1}^N\sup_{s\in[0,t]}\left | \Vis-\Wis \right|
& \le  
\frac{1}{N}\sum_{i=1}^N  \left | V^N_{i,0} - W^N_{i,0}\right |  +  \frac {2t}N +\int_0^t  \Lambda_s^N \,ds \\
& \hspace{10mm}  + \frac1{N^2}\sum_{i \neq j}^N \int_0^t  \left| K(X^N_{i,s}-X^N_{j,s})-K(Y^N_{i,s}-Y^N_{j,s})\right| \, ds  \\
& \hspace{40mm} \text{where} \quad 
\Lambda_s^N := \frac1N \sum_{i=1}^N\Lambda_{i,s}^N \, ds.
\end{split}
\end{equation}
Using equality~\eqref{eq:rope} we find with the notation $ \widehat{\rho}_s^{i,N}=\frac{1}{N-1}\sum_{j\neq i}\delta_{\Yjs}$,
\begin{align*}
\frac1{N^2}\sum_{i \neq j}^N  & \bigl | K(X^N_{i,s}-X^N_{j,s})-K(Y^N_{i,s}-Y^N_{j,s})\bigr|  \\
&   \leq \frac{2}{N^2}\sum_{i\neq j}^N   \indiq_{\left|Y^N_{i,s}-Y^N_{j,s}\right|\leq 2|X^N_{i,s}-Y^N_{i,s}|} 
=\frac{2(N-1)}{N^2}\sum_{i=1}^N  \int_\R  \indiq_{\left|Y^N_{i,s}-y\right|\leq 2|X^N_{i,s}-Y^N_{i,s}|}
\widehat{\rho}_s^{i,N}(dy)\\
&\leq\frac2N \sum_{i=1}^N  \int_\R  \indiq_{|Y^N_{i,s}-y |\leq 2|X^N_{i,s}-Y^N_{i,s}|}\rho_s(dy) 
 +\frac2N \sum_{i=1}^N \left| \int_\R  \indiq_{|Y^N_{i,s}-y |\leq 2|X^N_{i,s}-Y^N_{i,s}|}(\widehat{\rho}_s^{i,N}-\rho_s)(dy)  \right| \\
& \le \frac8N \,  \| \rho_s \|_\infty  \sum_{i=1}^N |X^N_{i,s}-Y^N_{i,s}| +  \Gamma^N_s
\end{align*}
\hspace{1mm} \text{where} \quad 
\begin{equation} \label{def:Gamma}
\Gamma^N_s := \frac2N \sum_{i=1}^N \Gamma_{i,s}^N, \qquad
\Gamma_{i,s}^N:= 
\sup_{u\in \R_+}\left | \int_{\R}  \indiq_{\left|Y^N_{i,s}-y\right|\leq u}(\widehat{\rho}_s^{i,N}-\rho_s)(dy)\right |.
\end{equation}
Then, $\beta(t) := \frac1N \sum_{i=1}^N \bigl( \sup_{s\in[0,t]} | \Xis-\Yis |+  \sup_{s\in [0,t]} | \Vis - \Wis  | \bigr)$ satisfies the integral inequality
\[
\beta(t) \le \beta(0) +  \int_0^t \biggl( \bigl( 1 + 8 \| \rho_s \|_\infty \bigr) \beta(s) + 
\Lambda_s^N+2 \, \Gamma_s^N+\frac2N  \biggr)\,ds.
\]
An application of Gr\"onwall's Lemma leads to
\begin{equation*}
\beta(t)
 \leq e^{\bigl(t+8\int_0^t\left\| \rho_s \right\|_{\infty} \,ds\bigr)} \,\biggl( \beta(0) + \int_0^t  e^{-s} \Bigl(\Lambda_s^N+2 \, \Gamma_s^N+\frac2N\Bigr)\, ds \biggr).
\end{equation*}
Taking the expectation and using the symmetry of the laws of the of $(\Xit,\Vit)_{i=1,\ldots,N}$ and $(\Yit,\Wit)_{i=1,\ldots,N}$ we find
\begin{multline} \label{eq:Coup}
\E\Bigl[ \sup_{s\in[0,t]} \left | X_{1,s}^N-Y_{1,s}^N \right|+\left | V_{1,s}^N-W_{1,s}^N \right| \Bigr] \, \leq e^{\bigl(t+8\int_0^t\left\| \rho_s \right\|_{\infty} \,ds\bigr)} \,  \\
\biggl( E\Bigl[  \left | X_{1,0}^N-Y_{1,0}^N \right|+\left | V_{1,0}^N-W_{1,0}^N \right| \Bigr] 
 +   \int_0^t e^{-s}
\Bigl(\E \bigl[\Lambda_{1,s}^N+2 \, \Gamma_{1,s}^N \bigr]+\frac2N\Bigr)\, ds \biggr).
\end{multline}
We will  bound the expectation of stochastic terms appearing in the r.h.s. with the help of Lemma~\ref{lem:(De)-Poiss}.

\medskip
{\sl Step 2. Conclusion of the proof.}

We recall that $ \widehat{\rho}_t^{i,N}=\frac{1}{N-1}\sum_{j\neq i}\delta_{\Yjt}$ is the empirical measure associated to the $(\Yit)_{ 2 \le i \le N}$, and is then  independent of $Y^N_{1,t}$.  By the definition~\eqref{def:Gamma} of $\Gamma$ and Lemma~\ref{lem:(De)-Poiss}, we have
\begin{align*}
\E\left[\Gamma^N_{1,t}\right]
& =\E\left[\sup_{u\geq 0}\left|\int \indiq_{|Y_{1,t}^N-y|\leq u}\bigl(\widehat{\rho}^{1,N}_t-\rho_t\bigr)(dy)\right|\right], \\
& =
 \E \Biggl[ \E\left[\sup_{u\geq 0}\left|\int \indiq_{|Y_{1,t}^N-y|\leq u}\bigl(\widehat{\rho}^{1,N}_t-\rho_t\bigr)(dy)\right|  \bigg| \, Y^N_{1,t}\right] \Biggr] \leq \frac{3}{\sqrt N},
\end{align*}

Moreover, using again the fact that the $(\Yit)_{1\le i \le N}$ are i.i.d and that $\| K\|_\infty \le 1/2$, we find for $N \ge 2$\begin{align*}
\E\bigl[\Lambda_{1,t}^N\bigr]
&=  \E\Biggl[\E\biggl[ \biggl | \frac{1}{N-1}\sum_{j=2}^{N}K( Y_{1,t}^N-Y^N_{j,t})- \int_{ \R^2 } K( Y_{1,t}^N-x)\mu_s(dx,dv) \biggr|  \; \bigg| \, Y^N_{1,t} \biggr] \Biggr], \\
& \le \E\Biggl[  \biggl( \V\biggl[   \frac{1}{N-1}\sum_{j=2}^{N}K( Y_{1,t}^N-Y^N_{j,t})    \; \bigg| \, Y^N_{1,t} \biggr]  \biggr)^{1/2} \Biggr] 
 \le   \frac1{2\sqrt{N-1}} \le  \frac1{\sqrt {2N}}.
\end{align*}
Then applying these results  to equation~\eqref{eq:Coup} leads  to 
\[
\E\Bigl[ \sup_{s\in[0,t]} \left | X_{1,s}^N-Y_{1,s}^N \right|+\left | V_{1,s}^N-W_{1,s}^N \right| \Bigr] \, \leq e^{\bigl(t+8\int_0^t\left\| \rho_s \right\|_{\infty} \,ds\bigr)} 
\biggl( \frac1{\sqrt N} ( 1 + 2 \times 3)   +\frac2N  \biggr),
\]
which leads to the claimed bound.

%
%
%
%
%
%

\section{Proof of Proposition ~\ref{prop:DevBound}} \label{sec:DevBound}
We distinguish here the deviation upper bounds for $\sup_{t\in[0,T]}\left\|\mu^N_{Y,t}\right\|_{\infty,\ep}$ because such a result have an interest by itself. Moreover, we will see it is used in the proof of the concentration inequalities. In the present proof,  we apply the following strategy: we first find prove concentration inequalities for the quantity $\bigl\|\mu^N_{Y,t}\bigr\|_{\infty,\ep}$ at a fixed time $t$, then for the variation of this quantity on small time intervals, and finally conclude by mixing both estimates in an optimal way.
That proof could be extended to dimension larger than $1$, but we prefer to restrict here to the case of interest.

\subsection{A uniform deviation upper bound for Binomial variables}
We begin with a general result about deviation upper  bounds for Binomial variable.		
\begin{lem}\label{lem:DevBern}
Let $p \in [0,1]$. If $X$ is binomial variable of parameter $(N,p)$, then  for any $\alpha>0$:
\begin{equation*}
\Prob\left( |X-Np| \geq N\alpha \right)\leq 2e^{-2\alpha^2 N}.
\end{equation*}
\end{lem}

\begin{proof}
We may write $X = \sum_{i=1}^N  X_i$, with a family $(X_i)_{i \le N}$  of  $N$ i.i.d. Bernoulli variables of parameter $p$. For all $\lambda>0$ we have:
\begin{equation*}
\E\bigl[  e^{\lambda X} \bigr] =\prod_{i=1}^N\E\bigl [ e^{\lambda X_i} \bigr ]
 =\left (  \E\bigl[  e^{\lambda X_1} \bigr] \right )^N
 = \left ( 1-p +pe^{\lambda} \right )^N.
\end{equation*}
Using Markov's inequality we get:
\begin{equation*}
\Prob\left ( X \geq N(p+\alpha) \right ) \leq \left ( \frac{1-p +p e^\lambda}{e^{\lambda (p+\alpha)}} \right )^N
 =\left((1-p)e^{-\lambda (p+\alpha)} +p e^{\lambda (1-(p+\alpha))}\right)^N,
\end{equation*}
The optimal $\lambda$ turns out to be
$ln\left(\frac{(p+\alpha)(1-p)}{(1-(p+\alpha))p}\right)$, and we get after some calculations:
\[\Prob\left ( X \geq N(p+\alpha) \right ) \leq e^{- N g_p(\alpha)},
\]
where
\[
 g_p(\alpha) := H\left ( \mathcal{B}(p+\alpha) \big|\mathcal{B}(p) \right )=(1-(p+\alpha))ln\left ( \frac{1-(p+\alpha)}{1-p}\right )+(p+\alpha)ln\left ( \frac{p+\alpha}p  \right ),
\]
is the relative entropy with respect to $\mathcal{B}(p)$. 
By the properties of relative entropy, we have $g_p(0)=0$ and $g_p(\alpha)>0$ if $\alpha \neq 0$. So $g'_p(0)=0$. Moreover a straight calculation gives for all $\alpha$ such that $p+\alpha < 1$:
\begin{equation*}
g''_p(\alpha)=\frac{1}{p+\alpha}+ \frac{1}{1-(p+\alpha)} \ge 4,
\end{equation*}
since for all $x \in (0,1)$, $x(1-x)\leq\frac{1}{4}$. Using Taylor's formula with integral rest at order 2 we have:
\begin{equation*}
g_p(\alpha) = g_p(0)+g'_p(0)\alpha+\int_{0}^\alpha g_p''(u)(\alpha-u)du  \geq 2\alpha^2.
\end{equation*}
So combining all these estimates we get for $\alpha + p <1$:
\begin{equation*}
\Prob\left ( X \geq N(p+\alpha) \right )   \leq e^{- Ng_p(\alpha)} \leq e^{-2N\alpha^2}.
\end{equation*}
It is still valid for $\alpha+ p=1$: just pass to the limit, and there is nothing to prove for $p+\alpha>1$, since $X$ cannot be larger than $N$.
Moreover, since
$\Prob\left ( X \leq N(p-\alpha) \right ) =\Prob\left ( N-X \geq N\left ( (1-p)+\alpha \right ) \right )$, an 
application of  the above bound to the $ \mathcal B(N,1-p)$-Binomial variable  $N-X$ leads to
\begin{equation*}
\Prob\left ( X \leq N(p-\alpha) \right )\leq e^{-2N\alpha^2},
\end{equation*}
and this concludes the proof.
\end{proof}

\subsection{Concentration inequalities at fixed time}		

Thanks to lemma~\ref{lem:DevBern} we are able to give some concentration inequalities for the empirical measure $\rho^N=\frac{1}{N}\sum_{i=1}^N\delta_{Y_i}$, for  i.i.d.r.v. $(Y_i)_{i _le N}$. 

\begin{lem} \label{lem:DevFixt}
Let $\alpha,\ep>0$. Assume that $(Y_1,\ldots,Y_N)$ are $N$ independent random variables, all with law $\rho \in L^\infty$ (we identify the law and its density). Assume also that $\rho$ has an exponential moment of order $\lambda>0$: $\MM_\lambda(\rho):= \int e^{ \lambda |y|} \rho(dy)  < + \infty$. We denote by $\rho^N := \frac1N \sum_{i} \delta_{Y_i}$ the associated empirical measure. Then for any $\ep, \alpha >0$,
\begin{equation*}
\Prob \bigl (  \| \rho^N  \|_\infep \geq  \| \rho \|_\infty +\alpha \bigr ) 
 \le 
 \biggl(  \frac{ 4  \| \rho\|_\infty N (\ep \alpha)}{\lambda}  + 2
 +  N \MM_\lambda(\rho) \biggr)\,  e^{-2N(\ep \alpha)^2}
\end{equation*}
\end{lem}

\begin{proof}

{\sl Step 1. A first bound valid on compact subset.}
For any $0<\delta<\ep$ we set $k = \left\lfloor  \frac R {2\delta}\right\rfloor +1$. It is clear that for all $x \in [-R,R]$, there exists $\ell \in \left \{ -k,\cdots, k \right \}$ such that $ B(x,\ep)\subset B(2 \ell \delta,\ep + \delta)$. It implies
\begin{align*}
\Prob \biggl (\sup_{x\in[-R,R]} \frac{\rho^N [ B(x,\ep) ]}{2\ep}  \geq  \| \rho \|_{\infty} & +\alpha \biggr)
\leq 
\Prob\left (\sup_{\ell=-k,\cdots,k}  \rho^N \bigl[ B(2 \ell \delta ,\ep + \delta) \bigr] \geq  2\ep \bigl( \| \rho  \|_\infty +\alpha \bigr) \right )\\
&\leq 
\sum_{\ell=-k}^k \Prob\Bigl(   \rho^N \left [ B(2  \ell \delta ,\ep + \delta) \right ] \geq  2 (\ep + \delta) \| \rho \|_\infty + 2 \bigl(\alpha  \ep - \delta \| \rho \|_\infty  \bigr) \Bigr).
\end{align*}
Since for any $\ell$, $N\rho^N\bigl(B(2\ell \delta,\ep + \delta)\bigr)$ is a Binomial variable of parameter $N$  and $p_\ell=\int_{B(2  \ell \delta ,\ep + \delta)}\rho_t(dx) \le 2 (\ep + \delta) \| \rho_t \|_\infty$, we may apply Lemma~\ref{lem:DevBern} and bound each term in the r.h.s. by $\exp\bigl(- 8 N (\alpha \ep - \delta \|  \rho \|_\infty)^2 \bigr)$. By the definition of $k$, and provided that $\ep \alpha  \ge \delta \| \rho \|_\infty$ it leads to
\[
\Prob\left (\sup_{x\in[-R,R]}  \frac{\rho^N \left [ B(x,\ep) \right ]}{2\ep} \geq \left \| \rho  \right \|_{\infty}+\alpha \right )
\leq  
\Bigl(\frac R \delta + 2 \Bigr) e^{- 8 N (\alpha \ep - \delta \|  \rho \|_\infty)^2}.
\]
We now have to choose $\delta$ in order to minimize the right hand side. The particular choice $\delta \| \rho\|_\infty = (\ep \alpha)/2$ satisfies the previous restriction and already provides an interesting bound: 
\begin{equation} \label{bound:locinfdis}
\Prob\left (\sup_{x\in[-R,R]}  \frac{\rho^N \left [ B(x,\ep) \right ]}{2\ep} \geq \left \| \rho  \right \|_{\infty}+\alpha \right )
\leq  
\biggl(   \frac{ 2 R  \| \rho\|_\infty}{\ep \alpha}  + 2 \biggr)\,  e^{-2N(\ep \alpha)^2}.
\end{equation}

\medskip
{\sl Step 2. Extension to the whole space.}
It is clear that
\begin{align*}
 \Prob \bigl (  \| \rho^N  \|_\infep \geq  \| \rho \|_\infty +\alpha \bigr ) 
 & \leq  
 \Prob\left ( \sup_{x\in [-R,R]}\frac{\rho^N \left [ B(x,\ep) \right ]}{2\ep} \geq \left \| \rho_t \right \|_{\infty}+\alpha \right )
 + \Prob \bigl ( \exists \, i \le N,  |Y_i| > R  \bigr ) \\
 & \le 
 \biggl(   \frac{ 2 R  \| \rho\|_\infty}{\ep \alpha}  + 2 \biggr)\,  e^{-2N(\ep \alpha)^2}
 +  N \MM_\lambda(\rho) e^{- \lambda R},
\end{align*}
where we used~\eqref{bound:locinfdis} to bound the first term in the r.h.s, and a simple application of Chebyshev's inequality to bound the second term: precisely that 
$ \Prob \left ( \exists \, i \le N,  |Y_i| > R  \right ) 
\le N \Prob \left (  |Y_1| > R  \right )  \le N \MM_\lambda(\rho) e^{- \lambda R}$. 
We choose $R=\frac2 \lambda N(\alpha\ep)^2$, and finally get:
\begin{equation*}
 \Prob \bigl (  \| \rho^N  \|_\infep \geq  \| \rho \|_\infty +\alpha \bigr ) 
 \le 
 \biggl(  \frac{ 4  \| \rho\|_\infty N (\ep \alpha)}{\lambda}  + 2 \biggr)\,  e^{-2N(\ep \alpha)^2}
 +  N \MM_\lambda(\rho) e^{-2N(\ep \alpha)^2},
\end{equation*}
and the result is proved.
\end{proof}

\subsection{Conclusion of the proof of Proposition~\ref{prop:DevBound}}

We begin with the following corollary of Proposition~\ref{prop:MomDev}, which will be useful in the sequel
\begin{cor} \label{cor:dev_Dt}
Let be $(Y_t,W_t)_{t\in [0,T]}$ be a solution of \ref{eq:NLSDE} for some initial condition $(Y_0,W_0)$ of law $f_0$, having an exponential moment of order $\lambda>0$. We  define $c_\lambda := \frac52  + \frac1 \lambda  \ln \Bigl( 
\E \bigl[ e^{\lambda |W_0|} \bigr] \Bigr)$ and denote by $f_t$ the law of $(Y_t,W_t)$. 
For $0 \le s < t \le s + \min \Bigl(\frac1{16}, \lambda^{-2} \Bigr)$, and $\beta>0$, it holds:
\[
 \Prob \biggl( \,
\sup_{s \le u \le t} |Y_u - Y_s| \ge (t-s) (c_\lambda + \beta)
 \biggr)
 \le   e^{- \frac \beta 2 \min \bigl( \beta, \lambda \bigr)}.
\]

And if the $(Y_i^N,W_i^N)$ are $N$ independent copies of the previous process, with the same notation
\[
 \Prob \Biggl( \,\frac1N \sum_{i=1}^N
\sup_{s \le u \le t} |Y_{i,u}^N - Y_{i,s}^N| \ge (t-s) (c_\lambda + \beta)
 \Biggr)
 \le   e^{- \frac \beta 2 \min \bigl( \beta, \lambda \bigr) N}.
\]

\end{cor}

\begin{proof}
Using point $(iii)$ of the Proposition~\ref{prop:MomDev}, and Chebyshev's inequality, and the definition of $c_\lambda$ we get for any $0<\lambda' \le \lambda$:
\begin{align*}
\Prob \Bigl( 
\sup_{s \le u \le t} |Y_u - Y_s| \ge (t-s) (c_\lambda + \beta)
 \Bigr) 
 & \le
e^{- \lambda' (c_\lambda + \beta) } 
\E \Bigl[ e^{\lambda' (t-s)^{-1} \sup_{s \le u \le t} |Y_u - Y_s| }  \Bigr] \\
& \le
\, e^{- \lambda' (c_\lambda + \beta) }  
e^{\frac12 \lambda'(5 + \lambda')  }\,  \E \Bigl[ e^{\lambda' |W_0|} \Bigr] \\
& \le  \,  e^{- \lambda' \beta + \frac 12 (\lambda')^2 } 
\E \Bigl[ e^{\lambda  |W_0|} \Bigr]^{-\frac {\lambda'} \lambda}
\E \Bigl[ e^{\lambda' |W_0|} \Bigr] 
 \le   e^{- \lambda' \beta + \frac12 (\lambda')^2 }.
\end{align*}
Optimization in $\lambda'$ leads to the particular choice $\lambda' = \beta $, when $\beta \le \lambda$, and to the choice $\lambda'= \lambda$ otherwise. If  we use $\beta - \frac \lambda 2 \ge \frac \beta 2$ in the later case, we obtain the expected bound.

The proof of the second bound involving $N$ independent copies follows the same lines: by independence 
\[
 \Prob \biggl( \,\frac1N \sum_{i=1}^N
\sup_{s \le u \le t} |Y_{i,u}^N - Y_{i,s}^N| \ge (t-s) (c_\lambda + \beta)
 \biggr) \le \biggl(  e^{- \lambda' (c_\lambda + \beta) } 
\E \Bigl[ e^{\lambda' (t-s)^{-1} \sup_{s \le u \le t} |Y_u - Y_s| }  \Bigr] \biggr)^N,
 \]
and the conclusion follows with the same optimization on $\lambda'$.
\end{proof}

We are now in position to prove Proposition~\ref{prop:DevBound}. 
We fix $\gamma >0$, define $\alpha := \frac \gamma 2$, and recall from Corollary~\ref{cor:dev_Dt} the notation $c_\lambda= \frac52 + \frac1 \lambda \ln \Bigl( \E \bigl[ e^{\lambda |W_0|} \bigr] \Bigr)$ together with $\kappa_t := \sup_{0 \le s \le t} \| \rho_s\|_\infty$.
\[
\beta:= \sqrt N \ep \gamma \, 
\max \Bigl( 1, \frac{\sqrt N \ep \gamma} \lambda \Bigr),
 \qquad
\Delta t := \frac {\ep \alpha} {(\kappa_t + \alpha)(c_\lambda + \beta)}.
\] 
Remark that $\beta$ satisfies $\frac \beta 2 \min \bigl(\beta , \lambda \bigr) = 2 N (\ep \alpha)^2$
and that  $\Delta t$  is always smaller than $\min \bigl(\frac1 {16}, \lambda^{-2} \bigr)$ by the assumptions in Proposition~\ref{prop:DevBound}, so that we may  apply Corollary~\ref{cor:dev_Dt}. 
We then choose  $K = \bigl\lfloor \frac t{\Delta t} \bigr\rfloor +1$, define $t_k= k \Delta t$ for all $0  \le k \le K$ (remark that $t_K \ge t$), and   
 the two following events $\Omega^1$ and $\Omega^2$ as:
\begin{align*}
\Omega^1 & =\left \{ \exists \;0 \le k \le K-1, \;  \ \left \| \rho^N_{t_k}  \right \|_{\infty,\ep+ \Delta t(c_\lambda+\beta)} > \kappa_t + \alpha \right \}, \\
\Omega^2 & =\left \{ \exists \; 0 \; \le k \le K-1,\; \exists \; i \le N, \ \sup_{s \in [t_k,t_{k+1}]} \bigl|\Yis -Y^N_{i,t_k} \bigr| >\Delta t( c_\lambda+\beta)  \right \}.
\end{align*}

If the events $\Omega_1^c$ and $\Omega_2^c$ are realized, then for any $0 \le s \le t$, we choose $k$ such that $s \in [t_k, t_{k+1})$ and get for any $x \in \R$
\begin{align*}
\rho^N_s\bigl( B(x,\ep) \bigr) & \le \rho^N_{t_k}\bigl( B(x,\ep+ \Delta t (c_\lambda + \beta) \bigr) 
 \le 
2  (\kappa_t + \alpha) \bigl( \ep+ \Delta t (c_\lambda + \beta) \bigr)\\
& = 2  \ep  (\kappa_t + \alpha) \biggl( 1 + \frac \alpha{(\kappa_t + \alpha)} \biggr) = 2 \ep (\kappa_t + \gamma).
\end{align*}
The last equalities follow from the definition of $\alpha$ and $\Delta_t$. 
It means that if $\Omega_1^c$ and $\Omega_2^c$ are realized, then 
\[
\sup_{s \in [0,t]} \bigl\| \rho^N_s \bigr\|_\infep \le \kappa_t + \gamma.
\]
Next, we can bound $\Prob(\Omega_1)$ and $\Prob(\Omega_2)$ with the help of Lemma~\ref{lem:DevFixt}, Corollary~\ref{cor:dev_Dt}, and point $(ii)$ of Proposition~\ref{prop:MomDev} (on the control of the exponential moments of $\rho_t$): it leads precisely to
\begin{align*}
\Prob \biggl(  \sup_{s \in [0,t]} \bigl\| \rho^N_s  & \bigr\|_\infep \ge \kappa_t + \gamma \biggr) 
\le \Prob \bigl(\Omega_1 \bigr) + \Prob \bigl(\Omega_2 \bigr) \\
& \le 
\sum_{k=0}^K \biggl(  \frac{ 4  \| \rho_{t_k} \|_\infty  N (\ep \alpha)}{\lambda}  + 2
 +  N \MM_\lambda(\rho_{t_k}) \biggr)\,  e^{-2N(\ep \alpha)^2}
+
8\, N K\,  e^{- \frac \beta 2 \min \bigl(\beta , \lambda \bigr)} \\
& \le 
\biggl( \frac t{\Delta t} + 1 \biggr)\biggl(  \frac{ 4  \kappa_t N (\ep \alpha)}{\lambda}  + 2
 +  2N e^{\lambda(\frac{1}{2}+ \lambda)t}\,\MM_\lambda^{x,v}(f_0) + 8 N  \biggr)\,  e^{-2N(\ep \alpha)^2} \\
 & = 
 \biggl( \frac {(2 \kappa_t + \gamma)(c_\lambda + \beta) t }{\ep \gamma}  + 1 \biggr)\biggl(  \frac{ 2 \kappa_t  (\ep \gamma)}{\lambda}  
 +  2e^{\lambda(\frac{1}{2}+ \lambda)t}\,\MM_\lambda^{x,v}(f_0) + 10  \biggr)\,N \,   e^{- \frac12 N(\ep \gamma)^2}.
\end{align*}
Using that by assumption $\ep \gamma \le \frac 12 \kappa_t  $, and the expression of $\beta$, leads to with 
\begin{align*}
\Prob \biggl(  \sup_{s \in [0,t]} \bigl\| \rho^N_s  \bigr\|_\infep \ge \kappa_t + \gamma \biggr) 
& \le \bs{C_1} \biggl( 1 + t\frac{2 \kappa_t + \gamma} {\ep \gamma} \Bigl(  c_\lambda + \sqrt N \ep \gamma \max(1, \sqrt N \ep \gamma \lambda^{-1}\Bigr)\biggr)\,N \,   e^{- \frac12 N(\ep \gamma)^2},
\end{align*}
$\bs{C_1}  := 10 +  \kappa_t^2 \lambda^{-1}  
 +  2e^{\lambda(\frac{1}{2}+ \lambda)t}\,\MM_\lambda^{x,v}(f_0)$.
The conclusion follows using that  $ \lambda  N^{-1/2} \le \ep \gamma$ by assumption.

%
%
%
%
%
%
%
%
%
%
%
%
%
%
%
%
\section{Proof of Theorem ~\ref{thm:main}} \label{sec:main}

\subsection{MKW estimates on devitation between particle and coupled systems}

The first step of the proof is similar to the first step of the proof of Theorem~\ref{thm:MeaPC}.  Precisely we start with equation~\eqref{eq:MKE},  which reads now
\[
\begin{split}
\frac{1}{N}\sum_{i=1}^N\sup_{s\in[0,t]}\left | \Xis-\Yis \right| & \leq  \int_0^t\frac{1}{N}\sum_{i=1}^N|V_{i,s}^N-W_{i,s}^N|ds\\
\frac{1}{N}\sum_{i=1}^N\sup_{s\in[0,t]}\left | \Vis-\Wis \right|
& \le \int_0^t \frac{1}{N^2}\sum_{i,j=1}^N \left| K(X^N_{i,s}-X^N_{j,s})-K(Y^N_{i,s}-Y^N_{j,s})\right| \, ds +  \frac t{N-1} +\int_0^t  \Lambda_s^N \,ds, \\
 \text{where} \quad  
\Lambda_s^N := \frac1N  \sum_{i=1}^N  &\Lambda_{i,s}^N \, ds = 
\frac1N \sum_{i=1}^N  \biggl | \frac{1}{N-1}\sum_{j=1}^{N}K(Y^N_{i,s}-Y^N_{j,s})- \int_{\R \times \R } K(Y^N_{i,s}-x)\mu_s(dx,dv) \biggr|,
\end{split}
\]
since the initial conditions are now equal.
Remark that if we introduce $\sigma,\tau$ two independent random variables with uniform law on $\{1,2,\ldots,N \}$, then the sum involving $K$ becomes
\[
\frac{1}{N^2}\sum_{i,j=1}^N \left| K(X^N_{i,s}-X^N_{j,s})-K(Y^N_{i,s}-Y^N_{j,s})\right| 
= \E_{\sigma,\tau} \Bigl[ \left| K(X^N_{\sigma,s}-X^N_{\tau,s})-K(Y^N_{\sigma,s}-Y^N_{\tau,s})\right| \Bigr],
\]
where we emphasize that the expectation is taken only with respect to $(\sigma,\tau)$. So, we are in position to apply point $(iii)$ Lemma~\ref{lem:rope} ; \emph{i.e.\ }the part involving discrete uniform norms, and get:
\begin{align*}
\frac{1}{N}\sum_{i=1}^N\sup_{s\in[0,t]}\left | \Vis-\Wis \right|
 & \le  8 \int_0^t \| \rho^N_s \|_\infep \biggl( \frac1N \sum_{i=1}^N \left| X^N_{i,s} - Y^N_{i,s} \right| + \frac \ep 2 \biggr) \, ds  +  \frac t{N-1} +\int_0^t  \Lambda_s^N \, ds \\
 & \le \int_0^t \bigl( 1 + 8 \| \rho^N_s \|_\infep\bigr) 
 \biggl( \frac1N \sum_{i=1}^N \left| X^N_{i,s} - Y^N_{i,s} \right| + \frac \ep 2 + \frac1{N-1}
 + \sup_{u \le t} \Lambda_u^N \biggr) \,ds
\end{align*}
where $\rho^N_s$ is the empirical measure of the $(Y_{i,s}^N)_{1\le i \le N}$: $\rho^N_s := \frac1N \sum_{i=1}^N \delta_{Y_{i,s}^N}$.
Applying finally Gronwall's Lemma on the interval $[0,t]$ where the quantity $\sup_{u \le t} \Lambda_u^N$ may be considered as fixed, we get:
\begin{align}\label{eq:EstCoup}
\frac{1}{N}\sum_{i=1}^N   & \sup_{s\in [0,t]} \bigl( |\Xis- \Yis|+   |\Vis-\Wit \bigr) 
\nonumber \\
&  \hspace{15mm}  \le 
\biggl( \frac \ep 2 + \frac1{N-1} + \sup_{s \le t} \Lambda_s^N \biggr)\,
\exp \biggl( t + 8 \int_{0}^{t} \left \|\rho_s^N  \right \|_\infep \, ds \biggr), \\
& \text{where} \quad 
\Lambda_s^N := \frac1N  \sum_{i=1}^N
\biggl | \frac{1}{N-1}\sum_{j=1}^{N}K(Y^N_{i,s}-Y^N_{j,s})- \int_{\R \times \R } K(Y^N_{i,s}-x)\mu_s(dx,dv) \biggr|.
\label{def:Lambda}
\end{align}
We now focus on finding some concentration inequalities for the random variable $\sup_{t\in [0,T]}\Lambda_t^N$. In order to prove  concentration inequalities for these supremum in time, we follow the same steps as in the proof of Proposition~\ref{prop:DevBound}.
Once it is be done, we will combine them with the concentration inequalities on 
 $\sup_{t\in [0,T]}\left \|\rho_t^N  \right \|_{\infty,\ep}$ given by Proposition~\ref{prop:DevBound}, and we will obtain some deviation upper bounds for 
\[
\frac{1}{N}\sum_{i=1}^N   \sup_{s\in [0,t]} \left( |\Xis- \Yis|+ |\Vis-\Wit| \right).
\]

\subsection{Estimation of the fluctuations term $\Lambda_t^N$ at fixed time $t$}
We first establish the

\begin{lem} \label{lem:dev_G}
Let $(Y^N_i)_{1 \le i \le N}$ be  i.i.d. random variables, with a diffuse common law (\emph{i.e.\ }a law that does not charge any atom). For $i=1,\ldots,N$ define the random variable $\Lambda^N$ as follows:
\begin{equation*}
\Lambda^N= \frac1N \sum_{i=1}^N \biggl|\frac{1}{N-1}\sum_{j\neq i} K(Y^N_i-Y^N_j)- \E \bigl[ K(Y^N_i-Y^N_j) \big| Y^N_i \bigr] \biggr|
\end{equation*}
Then for all $\alpha>0$,
\begin{equation*}
\Prob\left( |\Lambda^N| \geq \alpha \right)\leq 2 N \, e^{-2\alpha^2 (N-1)}.
\end{equation*}
\end{lem}

\begin{proof}
{\sl Step 1. A calculation with $Y^N_1$ frozen.} To begin, we ``freeze'' $Y^N_1 = a$ and define
\[
\Lambda^N_1(a) := \frac{1}{N-1}\sum_{j\ge 2} K(a-Y^N_j)- \E \bigl[ K(a-Y^N_j)\bigr]. 
\]
By the definition~\eqref{eq:defK} of $K$, and since $\Prob(Y^N_j =a) =0$ by assumption the random variable $\zeta_j^a$ defined by:
\begin{equation*}
\zeta_j^a=K(a- Y^N_j)+\frac{1}{2}=
\begin{cases}
0 & \text{ si } a< Y^N_j \\ 
1/2 & \text{ si } a= Y^N_j \\ 
1 & \text{ si } a> Y^N_j
\end{cases}
\end{equation*}
is a Bernoulli variable with parameter $p_a :=  \Prob( Y^N_j >a)$. 
 But since $ p_a-\frac{ 1}{2} = \E[K(a-Y^N_j)]$, we have
$\sum_{j\ge 2}\zeta_j^a-(N-1)p_a =   \Lambda^N_1(a)$.
So, by an application of Lemma~\ref{lem:DevBern} to the binomial variable $\sum_{j \le 2} \zeta_j^a$:
\begin{equation*}
\Prob\left ( \left |\Lambda_1^N(a) \right |\ge \alpha  \right )
=
\Prob\biggl( \biggl |\sum_{j\neq i}\zeta_j^a-(N-1)p_a\biggr | \geq (N-1) \alpha  \biggr)
 \le 2 \, e^{-\alpha (N-1)\alpha^2},
\end{equation*}

\medskip
{\sl Step 2. Summing up on $N$.}
Using the notation introduced in the previous step, we can rewrite 
$
\Lambda^N=\frac{1}{N}\sum_{i=1}^N \left | \Lambda_i^N(Y^N_i) \right |,
$
and then
\begin{align*}
\Prob\left ( \Lambda^N \geq \alpha \right ) &\le \Prob\left (\sup_{i=1,\ldots,N} |\Lambda_i^N(Y^N_i)| \geq \alpha \right )
 \leq \sum_{i=1}^N \Prob\left ( |\Lambda_i^N(Y^N_i)| \geq \alpha \right ) \\
 &=  N\,  \E\Bigl[ \Prob\bigl ( |\Lambda_1^N(Y^N_1)| \geq \alpha \big| Y^N_1 \bigr) \Bigr],
\end{align*}
where we have used the fact that the variables $(Y^N_i)_{1 \le i \le N}$ are exchangeable.
But by  independence of $Y^N_1$ and $(Y^N_i)_{i \ge 2}$, we obtain using the previous step that $\Prob\bigl ( |\Lambda_1^N(Y^N_1)| \geq \alpha \big| Y^N_1 \bigr)  \le 2 \, e^{-2 (N-1)\alpha^2}$ and the conclusion follows.
\end{proof}

\subsection{Estimation on the supremum in time of the fluctuations term}

\begin{lem} \label{lem:DevFluVar}
Assume that $(Y^N_i,W^N_i)$ are $N$ independent copies of a solution $(Y,W)$ to~\eqref{eq:NLSDE} of initial law $f_0$ satisfying $\mathcal{M}^{x,v}_{\lambda}(f_0) < +\infty$ for some $\lambda >0$, and $\Lambda^N_t$ is defined by~\eqref{def:Lambda}. Then, with $\rho_s := \LL(Y_s)$  $\kappa_t:= \sup_{0 \le s \le t} \| \rho_s \|_\infty$ we have provided that $\ep \le \min \Bigl( \frac 1 {16}, \frac{\lambda}{2}, \lambda^{-2} \Bigr)$.
\begin{equation*}
\Prob\left ( \sup_{s\in [0,t]} | \Lambda_s^N| \geq \bs{C_t}  \ep \right )\leq 
(\ep +  t)   
\biggl( \frac 5 \ep  +  4 \kappa_t N  \lambda^{-1}  + \frac N  \ep  e^{\lambda t  \bigl(\frac 12 + \lambda \bigr) } \E\Bigl[ e^{\lambda(|Y_0|+ |W_0|) }\Bigr] \biggr)
e^{- 2 N  \ep^2},
\end{equation*} 
with $\bs{C_t} :=   36  + 80 \kappa_t + (1+ 3 \kappa_t)  \frac 8 \lambda \ln \E \Bigl[ e^{\lambda |W_0|} \Bigr]$.
\end{lem}

\begin{proof}
{\sl Step 1. Bounding the time fluctuations.}
Let $(Z_t)_{t \ge 0}=(Y_t,W_t)_{t \ge 0}$ be a solution to the non linear SDE~\eqref{eq:NLSDE}. We show in that step that for $N \ge 2$:
\begin{multline}  \label{eq:time_flu}
\sup_{s\in[t,t+\Delta t]}|\Lambda^N_s-\Lambda^N_t|  \leq \biggl( 16\, \left\|\rho^N_t\right\|_\infep + 4 \| \rho_t \|_\infty  \biggr)  \sup_{s \in [t,t+\Delta t]} \frac1N \sum_{i=1}^N \bigl|\Yis-\Yit \bigr|\\
 + 4 \, \left \| \rho^N_t \right \|_\infep 
 |s-t| \Bigl( \E \bigl[ |W_0| \bigr] +2 \Bigr)
 + 10\,  \left\|\rho^N_t\right\|_\infep \ep.
\end{multline}
Indeed, by the definition ~\eqref{def:Lambda} of $\Lambda^N_s$:
\begin{multline*}
|\Lambda_N^s-\Lambda_N^{t}|
 \leq \frac{1}{N(N-1)}\sum_{i=1}^{N}\sum_{j=1}^{N}\left | K(\Yis-\Yjs)-K(\Yit-\Yjt) \right | \\
 +\frac{1}{N}\sum_{i=1}^N\left | \int_\R K(\Yis-x)\rho_s(dx)- \int_\R  K(\Yit-x) \,\rho_t(dx) \right |
\end{multline*}
Using the second point of Lemma~\ref{lem:rope} with two copies of a vector of joint law $\frac1N \sum_i \delta_{(\Yit,\Yis)}$, we may bound the first term in the r.h.s. by
\begin{equation*}
\frac{1}{N(N-1)}\sum_{i=1}^N\sum_{j=1}^N\left | K(\Yis-\Yjs)-K(\Yit-\Yjt)  \right |
\leq 
8 \frac N{N-1} \left\|\rho^N_t\right\|_{\infty,\ep} \biggl( \frac{1}{N}\sum_{i=1}^N |\Yis-\Yit|+ \frac\ep 2 \biggr).
\end{equation*}
To estimate the second term in the r.h.s, we use the third point of Lemma~\ref{lem:rope}, applied to independent couples: The first one   with law $\frac1N \sum_{i} \delta_{(\Yit,\Yis)}$ and $(Y_t,Y_s)$. It leads to the estimate
\begin{multline*}
\frac{1}{N}\sum_{i=1}^N\left | \int_\R K(\Yis-x)\rho_s(dx)- \int_\R  K(\Yit-x) \,\rho_t(dx) \right | \\
 \leq 4 \biggl( \| \rho_t \|_\infty \frac1N \sum_{i=1}^N  |\Yis-\Yit| + 
 \left \| \rho_t^N \right \|_\infep  \Bigl( \E \bigl[|Y_s-Y_t| \bigr] + \ep/2  \Bigr)\biggr ).
\end{multline*}
Putting these two estimates together, using point $(iv)$ of Proposition~\ref{prop:MomDev}  in order to bound $\E\bigl[ |Y_t-Y_s|\bigr]$ and taking the supremum in time leads to~\eqref{eq:time_flu}.

\medskip
{\sl Step 2. Controlling the deviation.}
We define $t_k= k \ep$ for $0 \le k \le k_M := \lfloor t \ep^{-1} \rfloor$ and recall that by assumption $\ep \le \min \Bigl( \frac1{16},\frac{ \lambda}{2},  \lambda^{-2}\Bigr)$. Consider the events 
\begin{align*}
\Omega_1 &:= \bigl\{ \| \rho^N_{t_k} \|_\infep \le \kappa_t +1  \text{ for all } 0 \le k \le k_M \bigr\} , \\
\Omega_2 &:= \biggl\{ \sup_{s \in [t_k,t_k+\ep]} \frac1N \sum_{i=1}^N \bigl|\Yis-Y^N_{i,t_k} \bigr| \le \ep ( c_\lambda + 2 \ep) \text{ for all } 0 \le k \le k_M \biggr\}, \\
\Omega_3 &:= \bigl\{ | \Lambda^N_{t_k} | \le 4 \ep \text{ for all } 0 \le k \le k_M \bigr\}.
\end{align*} 
By Lemma~\ref{lem:DevFixt} and the point $(ii)$ of Lemma~\ref{prop:MomDev}, 
\[
\Prob(\Omega_1^c) 
\le (k_M+1) \biggl( 4 \kappa_t N \ep \lambda^{-1} + 2 + N e^{\lambda t  \bigl(\frac 12 + \lambda \bigr) } \E\Bigl[ e^{\lambda(|Y_0|+ |W_0|)}\Bigr] \biggr) e^{-2 N \ep^2}.
\] 
By Corollary~\ref{cor:dev_Dt}, $\Prob(\Omega_2^c) \le (k_M+1) e^{- 2 N \ep^2}$. 
By Lemma~\ref{lem:dev_G}, $\Prob(\Omega_3^c) \le 2 (k_M+1) \, e^{- 8 (N-1) \ep^2} \le 2 (k_M+1)\,e^{- 2 N \ep^2}$ if $N \ge 2$.
When the three above events are realized, we get with the help of the bound~\eqref{eq:time_flu}
\begin{align*}
\sup_{s \in [0,t]} | \Lambda^N_s| & \le \sup_{k=0, \ldots , k_M} | \Lambda^N_{t_k}| + 
\sup_{k=0, \ldots, k_M, \; s\in[t_k,t_k+\ep]}|\Lambda^N_s-\Lambda^N_{t_k}|  \\
& \le 4 \ep +  (  20 \kappa_t + 4 ) \ep (c_\lambda + \ep ) + 4 \,( \kappa_t + 1) 
 \ep \Bigl( \E \bigl[ |W_0| \bigr] +2 \Bigr)
 + 10\,  ( \kappa_t + 1 ) \ep \\
 & \le \bs{C'_t} \ep,
\end{align*}
with $\bs{C'_t} := 26 + 4 c_\lambda+ 4 \E [ |W_0|] +  \kappa_t  ( 20 c_\lambda +  4 \E [ |W_0|]+  30 )$. 
Using the expression of $c_\lambda$ and the fact that $ \lambda \E [|W_0|] \le \ln E\bigl[ e^{\lambda |W_0|}\Bigr]$, it is not difficult to show that $\bs{C_t'} \le \bs{C_t}$, where $\bs{C_t}$ is the constant introduced in Lemma~\ref{lem:DevFluVar}.
Moreover, gathering the three previous estimates on $\Prob(\Omega_i^c)$ for $i=1,2,3$, we get
\begin{align*}
\Prob \bigl(\Omega_1^c \cup \Omega_2^c \cup \Omega_3^c \bigr)  & \le \frac {\ep +t} \ep  
\biggl( 5 +  4 \kappa_t N \ep \lambda^{-1}  + N e^{\lambda t  \bigl(\frac 12 + \lambda \bigr) } \E\Bigl[ e^{\lambda(|Y_0|+ |W_0|)}\Bigr] \biggr) \, e^{-2 N \ep^2}, 
\end{align*}
which concludes the proof.
\end{proof}

\subsection{Conclusion of the proof of Theorem~\ref{thm:main}}
We now consider the two events
\[
\Omega_1 := \Bigl\{ \sup_{s\in [0,t]} | \Lambda_s^N| \geq \bs{C_t}  \ep \Bigr\} , \qquad
\Omega_2 :=  \Bigl\{ \sup_{s \in [0,t]} \bigl\| \rho^N_s  \bigr\|_\infep \ge \kappa_t + 1 \Bigr\},
\]
where $\bs{C_t}$ is the constant defined in~\ref{lem:DevFluVar}.
By Lemma~\ref{lem:DevFluVar},  and since  we also assume $\ep \ge \lambda N^{-1/2}$
\begin{align*}
\Prob \left ( \Omega_1^c \right ) & \le 
 (\ep +  t)    \biggl( \frac 5 \ep  +  4 \kappa_t N  \lambda^{-1}  + \frac N  \ep  e^{\lambda t  \bigl(\frac 12 + \lambda \bigr) } \E\Bigl[ e^{\lambda(|Y_0|+ |W_0|)}\Bigr] \biggr)
e^{- 2 N  \ep^2}, \\
& \le 
(\ep +t)     \lambda^{-1}  N^{\frac32} \biggl(    5+ 4 \kappa_t +     e^{\lambda t  \bigl(\frac 12 + \lambda \bigr) } \E\Bigl[ e^{\lambda(|Y_0|+ |W_0|)}\Bigr] \biggr)
e^{- 2 N  \ep^2}
\end{align*}
and by Proposition~\ref{prop:DevBound}, provided that $\ep  \le 5 \kappa_t \min\Bigl( \frac1{16}, \frac\lambda 2,  \lambda^{-2} \Bigr)$,  and since  we also assume $\ep \ge \lambda N^{-1/2}$, the following bound holds with $\bs{D_t}  := 10 +  \kappa_t^2 \lambda^{-1}  
 +  e^{\lambda(\frac12+ \lambda)t}\,\E\Bigl[ e^{\lambda(|Y_0| + |W_0|)} \Bigr]$
\begin{align*}
\Prob \bigl(  \Omega_2^c \bigr) 
&  \le \bs{D_t} \biggl( 1 + t\frac{2 \kappa_t + 1} {\ep } \Bigl(  c_\lambda + \sqrt N \ep  \max(1, \sqrt N  \ep \lambda^{-1} ) \Bigr)\biggr)\,N \,   e^{- \frac12 N\ep ^2}, \\
& \le \bs{D_t} \biggl( 1 + t\frac{2 \kappa_t + 1} {\ep } \Bigl(  c_\lambda +  N \ep^2 \lambda^{-1} \Bigl) \biggr)\,N \,   e^{- \frac12 N\ep ^2}, \\
& \le \bs{D_t} \lambda^{-1} \biggl( \ep + t (2 \kappa_t + 1)  \Bigl(  c_\lambda + \sqrt   N \ep \Bigr) \biggr)\,N^{\frac32} \,   e^{- 2 N \ep ^2}.
\end{align*}
If the two events are satisfied, then thanks to~\eqref{eq:EstCoup} we get
\begin{align*}
\frac{1}{N}\sum_{i=1}^N    \sup_{s\in [0,t]} \bigl( |\Xis- \Yis|+   |\Vis-\Wit| \bigr) &  \le 
\biggl( \Bigl(  \bs{C_t} + \frac 1 2 \Bigr) \ep  + \frac1{N-1}  \biggr)\,
\exp \biggl( t + 8 \int_{0}^{t} \left \|\rho_s  \right \|_\infty \, ds \biggr), \\
&  \le 
\biggl(   \bs{C_t} + \frac 1 2   + \frac2\lambda  \biggr)\,
\exp \biggl( t + 8 \int_{0}^{t} \left \|\rho_s  \right \|_\infty \, ds \biggr) \ep ,
\end{align*}
where we have used that $\ep \ge \lambda N^{-1/2} \ge \lambda \bigl(2(N-1)\bigr)^{-1}$ when $N \ge 2$. 
which concludes the proof since 
\[
\bs{B_t} :=  \bs{C_t} + \frac12 +  \frac2\lambda.
\]
Defining also 
\begin{align*}
\bs{A_t''}  := \lambda^{-1} \Biggl( \bs{D_t}\bigl(1+ (2 \kappa_t +1) c_\lambda \bigr) +5+  4 \kappa_t +     e^{\lambda t  \bigl(\frac 12 + \lambda \bigr) } \E\Bigl[ e^{\lambda(|Y_0|+ |W_0|)}\Bigr]  \Biggr) , \quad
\bs{A_t'}  := \lambda^{-1} \bs{D_t}(2 \kappa_t +1),
\end{align*}
we obtain from the previous bound
$
\Prob \bigl( \Omega_1^c \cup \Omega_2^c \bigr) \le 
(t + \ep) \bigl( \bs{A_t} + \bs{A_t'} \sqrt N \ep \bigr) N^{\frac32} e^{- 2 N \ep ^2},
$
and that concludes the proof. In particular it can be checked that $\bs{A''_t} \le \bs{A_t}$ where $\bs{A_t}$ is defined by
\begin{align*}
\bs{A_t} & := \lambda^{-1}  \biggl[12 +  \kappa_t^2 \lambda^{-1}  
 + 2\,  e^{\lambda(\frac{1}{2}+ \lambda)t}\,\E\Bigl[ e^{\lambda(|Y_0| + |W_0|)} \Bigr] \biggr]
\bigl(1+ (2 \kappa_t +1) c_\lambda \bigr),  \\
\bs{A_t'} & := \lambda^{-1}  \biggl[10 +  \kappa_t^2 \lambda^{-1}  
 +  e^{\lambda(\frac{1}{2}+ \lambda)t}\,\E\Bigl[ e^{\lambda(|Y_0| + |W_0|)} \Bigr] \biggr] (2 \kappa_t +1),\\
\bs{B_t} & :=   37 +  \frac2\lambda + 80 \kappa_t + (1+ 3 \kappa_t)  \frac 8 \lambda \ln \E \Bigl[ e^{\lambda |W_0|} \Bigr].
\end{align*}

%
%
%
%
%
\section{Appendix}

\begin{prop}\label{prop:app}
Let $K:(t,x)\in \R^+\times \R \mapsto K_t(x)\in \R$ be a function belonging to $L^\infty_{loc}\bigl(\R^+, C^k(\R)\bigr)$ for all $k \in \N$ and to $C^1(\R^+ \times \R^2)$. Consider the unique  (in the class of measures) solution to the following linear PDE 
\begin{equation}\label{eq:LVFp}
\partial_t f_t+v \, \partial_x f_t+K_t(x) \, \partial_v f_t= \partial_v(\partial_v f_t +v f_t),
\end{equation}
for an initial condition $f_0 \in \PPP(\R^2)$ satisfying $\partial_x^k \partial_v^l f_0 \in L^2(\R^2)$ for all $k,l \in \N$. Then $f \in C^1(\R^+ \times \R^2)$ and is even two times continuously differentiable in $(x,v)$.
\end{prop}

\begin{proof}
Differentiating equation~\eqref{eq:LVFp} $k$ times in $x$ variable, and $l$ times in $v$ variable leads to
\begin{multline} \label{eq:diff-kxlv}
\partial_t (\partial_x^k\partial_v^l f_t)+v \, \partial_x(\partial_x^{k}\partial_v^l f_t)+ K_t(x) \,\partial_v(\partial^k_x\partial^l_v f_t) -  \partial_v( v \, \partial_x^k\partial_v^l f_t ) - \partial_v^2 (\partial_x^k \partial^l_v f_t  ) \\
= - \sum_{k'=0}^{k-1}\binom{k}{k'} \partial_x^{k-k'}K_t(x) (\partial_x^{k'}\partial_v^{l+1}f_t)
- \partial_x^{k+1}\partial_v^{l-1}f_t  + l \, \partial_x^k\partial_v^l f_t,
\end{multline}
with the convention that a term containing a derivative with a negative power vanishes. 
Multiplying the above equation by $\partial_x^k\partial_v^l f_t$ and performing an integration by part, we obtain
\begin{align*}
\frac{d}{dt}\bigl \| \partial_x^k\partial_v^l f_t \bigr \|^2_2 
& \leq \bigl\| \partial_x^k\partial_v^l f_t \bigr\|_2  \bigl\|  \partial_x^{k+1}\partial_v^{l-1} f_t \bigl\|_2 
+C_{t,k}\sum_{k'=0}^{k-1}\binom{k}{k'}\bigl\|  \partial_x^{k'}\partial_v^{l+1}f_t \bigr\|_2 \bigr\| \partial_x^k\partial_v^l f_t \bigr\|_2 + \Bigl( l + \frac12 \Bigr) \bigl\| \partial_x^k\partial_v^l f_t \bigl \|^2_2, \\
& \text{where} \quad 
C_{t,k}=\sup_{k'=1,\cdots,k} \bigl\| \partial_x^{k'}K_t \bigr\|_\infty.
\end{align*}
Then summing these inequalities over the $k,l$ such that $k+l \le m$, we find
\[
\frac{d}{dt} H_m(t) \leq  \tilde C_{m,t} H_m(t),
\quad \text{where} \quad
H_m(t)=\sum_{k+l \le m} \left \| \partial_x^k\partial_v^l f_t \right \|^2_2,
\]
and the constant $\tilde C_{m,t}$ is locally bounded in time. Since by assumption $H_m(0) < \infty$ for any $m \in \N$, we get that for any  time $t \ge 0$,  $\sup_{s \in [0,t]} H_m(s) < \infty$.
Then by Morrey's inequality for $m$ large enough ($m=4$), $f\in L^\infty([0,t],C^2(\R^2))$, and so does $\partial_x^k\partial_v^l f$ for all $k,l\geq 0$. Using ~\eqref{eq:diff-kxlv}, we deduce that $\partial_t \partial_x^k\partial_v^l f\in L^\infty_{loc}(\R^+ \times \R^2)$, and in particular $\partial_x^k\partial_v^l f$ is continuous. So $f$ is two times continuously differentiable in $(x,v)$.  
And using finally~\eqref{eq:LVFp}, we see that $\partial_t f$ is itself continuous in all the variables, as a sum of continuous functions. This concludes the proof. 
\end{proof}

\end{document}